\documentclass{amsart}%
\usepackage{amsfonts, amsmath, amssymb, graphicx, epic, epsf}%
\usepackage{pstricks, pst-grad, pst-plot}

\input xy
\xyoption{all}

\newtheorem{theorem}{Theorem}
\theoremstyle{plain}

\newtheorem{corollary}{Corollary}

\newtheorem{proposition}{Proposition}
\newtheorem{remark}{Remark}
\numberwithin{equation}{section}

\DeclareMathOperator{\id}{Id}

\def\co{\colon\thinspace} 

\def\bbC{{\mathbb C}}

\def\bbZ{{\mathbb Z}}

\def\calA{{\mathcal A}}
\def\calC{{\mathcal C}}

\def\H{{\mathcal H_{a,h}^*}}

\def\calL{{\mathcal L}}

\begin{document}

\title{On the quantum filtration of the universal $sl(2)$ foam cohomology}
\author{Carmen Caprau}
\address{Department of Mathematics, California State University, Fresno, CA 93740 USA}
\email{ccaprau@csufresno.edu}
%\urladdr{http://zimmer.csufresno.edu/~ccaprau}

\date{}

\begin{abstract}
We investigate the filtered theory corresponding to the universal $sl(2)$ foam cohomology $\H$ for links, where $a, h \in \bbC$. We show that there is a spectral sequence converging to $\H$ which is invariant under the Reidemeister moves, and whose $E_1$ term is isomorphic to Khovanov homology. This spectral sequence can be used to obtain from the foam perspective an analogue of the Rasmussen invariant and a lower bound for the slice genus of a knot.
 \end{abstract}

\maketitle 
\section{Introduction}

The author constructed in \cite{CC1} an invariant for oriented links in $\mathbb{S}^3,$ which is a $\mathbb{Z} \oplus \mathbb{Z}$-graded module $\mathcal{H}^{*,*}$ over $R = \bbZ[i, a,h],$ where $i^2 = -1$ and $a$ and $h$ are formal parameters. One of the $\mathbb{Z}$ gradings is the \textit{cohomological grading} and the other is the \textit{quantum grading} given by degrees of polynomials. The construction was done via \textit{webs} and \textit{foams} (disoriented cobordisms) modulo local relations, along the lines of Bar-Natan's work~\cite{BN1}. The advantage of this approach is that it yields a theory which is properly functorial with respect to link cobordisms with no sign indeterminacy, and a clearer geometric picture of the construction.
  
The ring $R$ is graded by $\deg(a) = 4, \deg(h) = 2,$ and $\deg(1) = \deg(i) = 0.$ Setting $a, h \in \bbC$ such that $h^2 + 4a \neq 0$, and working over $\bbC$, one loses the bi-grading and obtains a $\bbZ$-graded theory with a filtration in place of the quantum grading. The purpose of the current paper is to investigate the filtered version of the universal $sl(2)$ foam cohomology $\H$ for links (the subscript in $\H$ refers to the fact that $a$ and $h$ are not parameters anymore but fixed complex numbers), motivated by earlier works which show that interesting results can be obtained by studying filtered link theories. The first works in this direction are those of Lee~\cite{L} and Rasmussen~\cite{R} followed by, for example,~\cite{Lo, T, W}. Our approach to the filtered theory is somewhat similar to that in~\cite{R, W}.

We assume familiarity with the construction in~\cite{CC1}, but to establish some notations we briefly recall a few concepts and statements about the universal $sl(2)$ foam cohomology.

We denote by $\textit{Foams}$ the category whose objects are webs and whose morphisms are $R$-linear combinations of foams, and we denote by $\textit{Foams}_{/\ell}$ the quotient category of $\textit{Foams}$ by the relations $\ell,$ that is, we mod out the morphisms of the category \textit{Foams} by the local relations $\ell$--these are the ``generalized" Bar-Natan relations enhanced by additional relations involving the disoriented 2-sphere. We draw foams with their source at the bottom and target at the top. The functor $ \textbf{Foams}_{/\ell} \to \textbf{R-Mod}$ used in the foam theory, taking us from the geometric picture to the algebraic picture, is strongly connected to the \textit{universal Frobenius system} of rank two (universal in the sense of Khovanov's work~\cite{Kh2}) defined on the graded $R$-module $\calA = R[X]/(X^2 - hX -a)$. With respect to the generators $1$ and $X$, with $\deg(1) = -1$ and $\deg(X) = 1$, the comultiplication $\Delta$ and multiplication $m$ corresponding to the algebra $\calA$ are homogeneous maps and are given by
\begin{eqnarray*}
&&\Delta(1)= 1\otimes X + X \otimes 1 - h 1 \otimes 1,\quad  \Delta(X) = X \otimes X + a 1 \otimes 1\\
&&m(1\otimes 1) = 1, m(1\otimes X) = m(X \otimes 1) = X, m(X\otimes X) = hX + a
\end{eqnarray*} 
and the counit and unit by $\epsilon(1) = 0, \epsilon(X) = 1, \iota(1) = 1.$ The TQFT corresponding to $\mathcal{A}$ factors through the quotient category of \textit{Foams} by the relations $\ell.$

The complex (living in the algebraic world) associated to an oriented link diagram $L$ is isomorphic to the complex whose objects are tensor powers of $\calA$ and whose differential is built up from the maps $\Delta$ and $m$ sprinkled by powers of $i.$ We denote this complex by $\calC^{*,*}(L)$ and its differential by $d.$ 

In this paper we let $a$ and $h$ be complex numbers such that $X^2 - hX -a = (X- \alpha) (X -\beta),$ where $\alpha, \beta \in \mathbb{C}, \alpha \neq \beta$ and we work over $\bbC.$  While the differential $d$ does not respect the quantum grading anymore, it cannot increase this grading (it is easy to see that if $v$ is a homogeneous element, then the quantum grading of every monomial in $\Delta(v)$ or $m(v)$ is less than or equal to that of $v$), giving rise to a $\bbZ$-graded theory with a filtration in place of the quantum grading. We denote the resulting complex by $\calC^*_{a,h}(L)$ and we refer to the corresponding link cohomology as the \textit{filtered foam cohomology}, denoted by $\H(L).$ 

A \textit{state} $\phi$ of $L$ is a labeling of its arcs by $\alpha$ and $\beta,$ and a state is called \textit{canonical} if the arcs belonging to the same component of $L$ have the same label. The author showed in~\cite{CC1} that the following result holds for the  \textit{filtered foam theory}.

\begin{theorem} \cite[Theorem 4]{CC1}
For any $n$-component link $L,$ and $a, h \in \bbC$ such that $h^2 + 4a \neq 0,$ the dimension of the cohomology group $\H(L)$ equals $2^n.$
\end{theorem}

Indeed, there is a one-to-one correspondence between the set of canonical states of $L$ and the set of all generators $h_{\phi}$ of $\H(L),$ which can be described as follows: each canonical state $\phi$ defines precisely one resolution $\Gamma_{\phi},$ which we refer to as the \textit{canonical resolution} of $L,$ obtained by resolving to the disoriented resolution all crossings at which the $\phi$-values of the two strands are different, and resolving to the oriented resolution all crossings at which the $\phi$-values of the two strands are equal. Denote by $h_{\phi}$ the induced generator in the cohomology group $\H(L)$ by the canonical resolution $\Gamma_{\phi}.$ All $h_{\phi}$ generate $\H(L).$ The reader may have observed the similarities with Lee's work in~\cite{L}, where she exhibits a bijection between the set of possible orientations of $L$ and the set of generators of her homology theory.
 
To this end, let $CKh^{*,*}(L)$ be the Khovanov complex~\cite{ BN0,Kh1} over $\bbC$ and denote by $Kh_{\bbC}^{*,*}(L)$ the resulting $\bbC$-Khovanov homology.

\begin{theorem}\label{main thm}
For each $a, h \in \bbC$ such that $h^2 + 4a \neq 0,$ there is a spectral sequence with $E_1 \cong Kh_{\bbC}^{*,*}(L^!)$ and converging to $\H(L)$ (where $L^!$ denotes the mirror image of $L$). All pages $E_j, j \geq 1,$ of this spectral sequence are invariants of the link $L.$
\end{theorem}
\label{cobordisms}

We prove the above theorem in section~\ref{spectral sequence}. In section~\ref{cobordisms} we show that the generators for $\H$ behave well under link cobordisms. Then it is easy to see that our spectral sequence can be used to obtain an analogue of the Rasmussen invariant $s(K)$ defined in~\cite{R}, which we state it below as a corollary. 

The quantum filtration $F$ on $\H(L)$ is bounded from above and below for any link $L$, and the highest and lowest filtration levels of $\H(L)$ are
\begin{eqnarray*}
s_{\text{max}}(L) &=&\text{min}\{k \, | \,  F^k \H(L) = \H(L) \} \\
s_{\text{min}}(L) &=& \text{min}\{k \, | \,  F^k \H(L) \neq 0  \}.
\end{eqnarray*}
For any knot $K,$ we define the invariant $\overline{s}(K) = \frac{1}{2}[s_{\text{max}}(K) + s_{\text{min}}(K)].$

\begin{corollary}
For any link $L$ in $S^3,$ we have that $s_{\text{max}}(L) \geq \chi_s(L).$ In particular, for any knot $K,$ $|\overline{s}(K) | \leq 2g_s(K),$ where $\chi_s(L)$ and  $g_s(L)$ are the \textit{slice Euler characteristic}  and the \textit{slice genus} of $L,$ respectively.

\end{corollary}
\begin{remark}
It follows from construction that $\overline{s}(K) = -s(K^!).$
\end{remark}

%%%%%%%%%%%%%%%%%%%%%%%%%%%%%%%%%%%%%%%%%%%%%%%%%%%
\section{The spectral Sequence}\label{spectral sequence}

We explain now the filtration grading of an arbitrary element $v \in \calC^*_{a,h}(L).$ Recall that $\text{deg}(1) = -1, \text{deg}(X) = 1$ for the generators $1$ and $X$ of the algebra $\calA.$ Then $\text{deg}(v_1 \otimes v_2 \otimes \dots \otimes v_m) = \text{deg}(v_1) + \text{deg}(v_2) + \dots + \text{deg}(v_m),$ for arbitrary $v_i \in \{1,X \}.$

An arbitrary $v \in \calC^*_{a,h}(L)$ is not homogeneous but can be written as $v = v_1 + v_2 + \dots + v_l,$ where $v_j$ is homogeneous for each $j.$ We define
\[\text{deg}(v): = \text{max} \{\text{deg}(v_j) \, \vert \, j = 1, \dots ,l \}. \]
For any $v \in \calC^r_{a,h}(L),$ its filtration grading is given by
$ q(v) = \text{deg}(v) -r + n_{+} - n_{-}, $
where $n_{+}$ and $n_{-}$ are the numbers of positive and negative crossings, respectively, in $L.$ 
Finally, set 
\[ F^k \calC^*_{a,h}(L) = \{v \in \calC^*(L) \, \vert \, q(v) \leq k \}. \]
This defines and increasing filtration $\{F^k \calC^*(L)\}$ on $\calC^*(L)$
\[\dots F^{k-1}\calC^*_{a,h}(L) \subset F^{k}\calC^*_{a,h}(L) \subset F^{k+1}\calC^*_{a,h}(L) \subset  \dots \,.\]
Note that the grading defined above says that $v \in  \calC^*_{a,h}(L)$ has filtration grading $q(v)$ if and only if $v \in F^{k}\calC^*_{a,h}(L)$ but $v \notin F^{k -1}\calC^*_{a,h}(L).$
 
 If $f: A \to B$ is a map between filtered chain complexes, we say that $f$ is a \textit{filtered map of degree} $i$ if $f(F^{k}A) \subset F^{k + i}B,$ and we say that $f$ \textit{respects the filtration} if $f$ is a filtered map of degree zero, i.e., $f(F^{k}A) \subset F^{k}B.$ 
 
The filtration $\{F^{k}\calC^*_{a,h}(L)\}$ on $\calC^*_{a,h}(L)$ induces a filtration $\{F^k \H(L)\}$ on $\H(L),$ called the \textit{quantum filtration}: an element of $\H(L)$ is in $F^k \H(L)$ if and only if it is represented by a cocycle in  $F^{k}\calC^*(L).$
 
 %%%%%%%%%%%%%%%%%%%%%%%%%%%%%%%%%%%%%%%%%%%%%%%%%%
We can now prove the first result of this paper.

\begin{proof} (of Theorem~\ref{main thm})
We have $dF^k  \subset F^k$ and since only finitely many vector spaces $\calC^*_{a,h}(L)$ are non-trivial, it follows that the filtration is bounded and there is a spectral sequence which converges to $\H(L).$ 

The underlying groups of $\calC^*_{a,h}(L)$ coincide with those of $CKh^{*,*}(L^!).$ Moreover, the differential on the $E_0$-page of this spectral sequence is the part of $d$ which preserves the $q$-grading, thus corresponds to the maps (up to multiplication by some powers of $i$)
\begin{eqnarray*}
\Delta'(1) &=& 1\otimes X + X \otimes 1 \hspace{1cm} m'(1 \otimes X) = m'(X \otimes 1) = X\\
\Delta'(X) &=& X \otimes X  \hspace{2.2cm} m'(1 \otimes 1) = 1, \,m'(X \otimes X) = 0, 
\end{eqnarray*} 
which are exactly the maps used in the construction of Khovanov's complex. Therefore the $E_0$-page of the spectral sequence is isomorphic to the complex $CKh^{*,*}(L^!),$ and the first statement follows.

To prove the second part of the theorem, we need to show that the spectral sequence is invariant under the Reidemeister moves. Let $D_1$ and $D_2$ be two link diagrams related by a Reidemeister move. In~\cite{CC1}, we have constructed (formal) chain maps $f \co [D_1] \to [D_2]$ and $g \co [D_2] \to [D_1]$ which form a homotopy equivalence. These maps induce chain maps 
\[\overline{f} \co \calC^*_{a,h}(D_1) \to \calC^*_{a,h}(D_2) \quad \text{and} \quad \overline{g} \co \calC^*_{a,h}(D_2) \to \calC^*_{a,h}(D_1).\]
Both $\overline{f}$ and $\overline{g}$ are filtered chain maps which respect the filtrations on $\calC^*_{a,h}$ (this can be seen in a case by case inspection and taking the grading shifts into consideration), inducing $\bbZ \oplus \bbZ$-graded chain maps 
\[\overline{f}_k \co E_k ^{*,*}(D_1) \to E_k ^{*,*}(D_2) \quad \text{and} \quad \overline{g}_k \co E_k ^{*,*}(D_2) \to E_k ^{*,*}(D_1), \, k \geq 0.  \]

Denote by $\psi_1 \co CKh^{*,*}(D_1^!) \to E_0^{*,*}(D_1)$ and $\psi_2 \co CKh^{*,*}(D_2^!) \to E_0^{*,*}(D_2)$ the isomorphisms of $\bbZ \oplus \bbZ$-graded chain complexes obtained in the proof of the first part of the theorem.

Then it is not hard to see that the following diagrams commute
\begin{displaymath}
\xymatrix @C=13mm@R=10mm{
CKh^{*,*}(D_1^!) \ar [r]^{F} \ar [d]_{\psi_1} & CKh^{*,*}(D_2^!) \ar[d]_{\psi_2}\\
E_0^{*,*}(D_1) \ar[r]^{\overline{f}_0} & E_0^{*,*}(D_2)
} \hspace{0.5cm}
\xymatrix @C=13mm@R=10mm{
CKh^{*,*}(D_2^!) \ar [r]^{G} \ar [d]_{\psi_2} & CKh^{*,*}(D_1^!) \ar[d]_{\psi_1}\\
E_0^{*,*}(D_2) \ar[r]^{\overline{g}_0} & E_0^{*,*}(D_1)
} 
\end{displaymath}
 where $F \co CKh^{*,*}(D_1^!) \to CKh^{*,*}(D_2^!)$ and $G \co CKh^{*,*}(D_2^!) \to CKh^{*,*}(D_1^!)$ are the homotopically inverse morphisms of chain complexes (induced by the homotopically inverse morphisms given at the topological level) constructed in the proof of the Invariance Theorem in~\cite{BN1}. 
 
 The morphisms $\overline{f}_0$ and $\overline{g}_0$ therefore form  a homotopy equivalence between the associated graded complexes, and they induce mutually inverse maps on homology. But the homologies of the associated graded complexes are contained in the $E_1$-pages, which implies that the morphisms 
 $\overline{f}_1$ and $\overline{g}_1$ are mutually inverse. We can conclude thus that each $\overline{f}_k \co E_k^{*,*}(D_1) \to E_k^{*,*}(D_2), \, k \geq 1$ is an isomorphism (by a well-known result; see e.q. the Mapping Lemma 5.2.4 in~\cite{We}).
\end{proof}

%%%%%%%%%%%%%%%%%%%%%%%%%%%%%%%%%%%%%%%%%%%%%%%%%%%%BEHAVIOR UNDER COBORDISMS
\section{Behavior under cobordisms}\label{cobordisms} 

The following result is implied by the author's work in~\cite{CC1}.
\begin{proposition}
Given a link cobordism $C \subset \mathbb{R}^3 \times [0, 1]$ between links $L_0$ and $L_1$, there is an induced filtered map $\calL_C \co \H(L_0)  \longrightarrow \H(L_1)$ of degree $-\chi(C)$, well defined under ambient isotopy of $C$ relative to its boundary.
\end{proposition}

We would like to know the behavior of the map $\calL_C$ with respect to the generators for $\H$. Recall that an \textit{elementary cobordism} corresponds to a Reidemeister move and a Morse move (creation/annihilation of a circle and saddle move), and that an arbitrary link cobordism $C\co L_0 \to L_1$ can be decomposed into a union of elementary cobordisms $C = C_1 \cup C_2 \cup \dots \cup C_l.$ The induced morphism $\calL_C \co \H(L_0)  \longrightarrow \H(L_1)$ is then the composition $\calL_{C_l} \circ \dots \circ \calL_ {C_1}.$ The above proposition says that $\calL_C$ is a filtered map of degree $-\chi(C)$ and that it depends only on the isotopy class of $C$ rel $\partial C.$ Thus it suffices to study the behavior of the map $\calL_C$ with respect to the generators of $\H$ only when $C$ is an elementary cobordism. 

Before we look into this, we recall that, as a Frobenius algebra, the ring $\calA$ decomposes as $\calA \cong R \oplus R.$ The orthogonal idempotent basis $\{ z_1, z_2\}$ of $\calA$ associated to this decomposition is given by
$ z_1 = \frac{X - \alpha}{\beta - \alpha}$ and $z_2 = \frac{X - \beta}{\alpha - \beta}, $
and the Frobenius algebra structure of $\calA$ with respect to this basis has the form
\begin{eqnarray*}
 \epsilon (z_1) = \frac{1}{\beta - \alpha}, \quad \epsilon (z_2) = \frac{1}{\alpha - \beta}, \quad  \iota(1) = z_1 + z_2,  \\
m(z_1 \otimes z_1) = z_1, \quad m(z_2 \otimes z_2) = z_2, \quad m(z_1 \otimes z_2) = m(z_2 \otimes z_1) =  0,\\
 \Delta(z_1) = \frac{1}{\beta - \alpha}(z_1 \otimes z_1), \quad  \Delta(z_2) = \frac{1}{\alpha - \beta}(z_2 \otimes z_2). 
 \end{eqnarray*}

We know that each canonical state $\phi$ of a link $L$ defines a canonical resolution $\Gamma_{\phi}$ which induces a \textit{canonical generator} $h_{\phi} \in \H(L).$ Each connected component of  a canonical state is a closed web containing $2l$ bivalent vertices (where $l \in \bbZ, l\geq 0$),  $l$ arcs labeled $\alpha$ and $l$ arcs labeled $\beta.$ All arcs labeled $\alpha$ are oriented the same way, but oppositely oriented to those labeled $\beta.$ From~\cite{CC1} we know that there are certain isomorphisms that can be used to `remove' all arcs labeled $\beta$ (or all arcs labeled $\alpha$) and `replace' each closed web by an oriented circle labeled $\alpha$ (or $\beta$), whose orientation is given by that of the arcs labeled $\alpha$ (or $\beta$) in the original graph, as explained below (note that these isomorphisms will introduce some powers of $i$ on the per-edge maps in the complex associated to the link diagram). 
\[\raisebox{-13pt}{\includegraphics[height=0.6in]{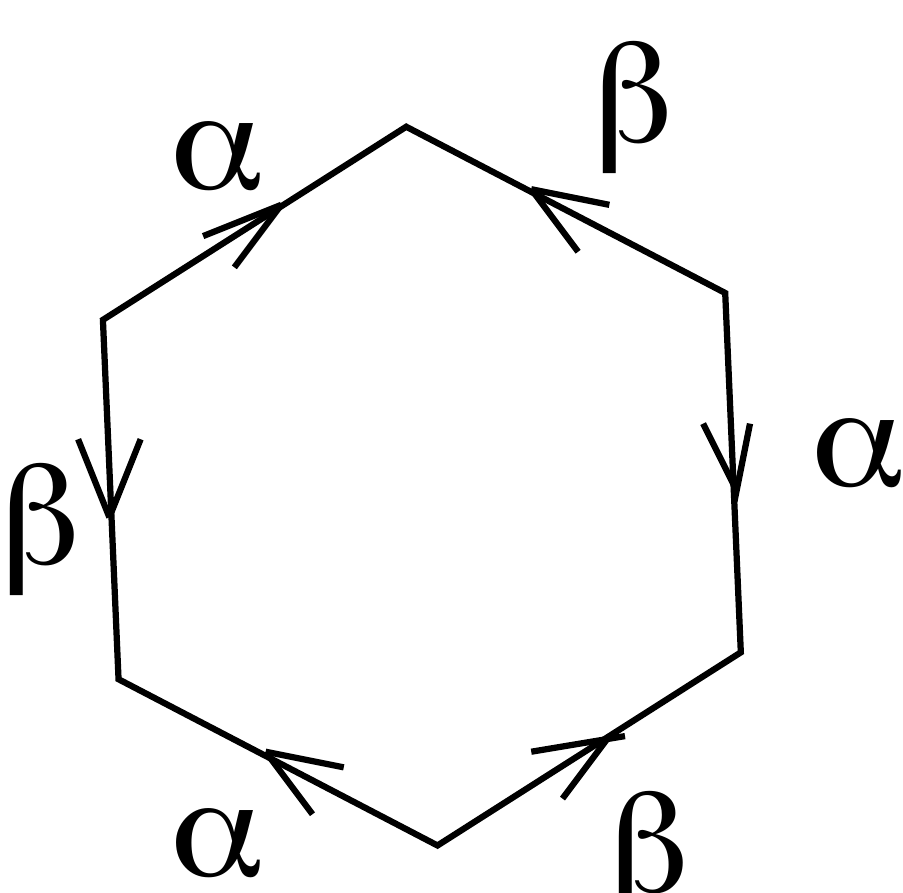}} \longrightarrow \raisebox{-13pt}{\includegraphics[height=0.5in]{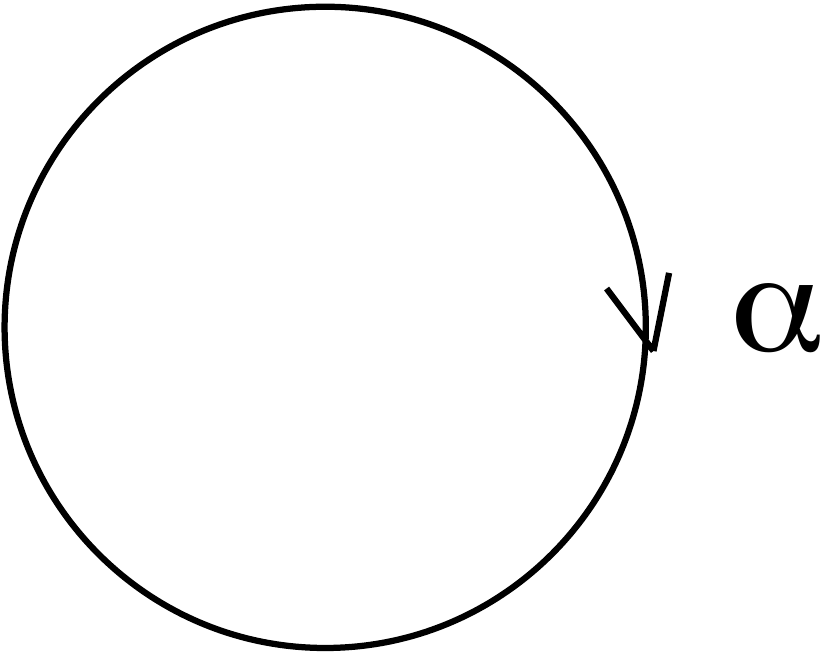}} \quad \text{or} \quad \raisebox{-13pt}{\includegraphics[height=0.5in]{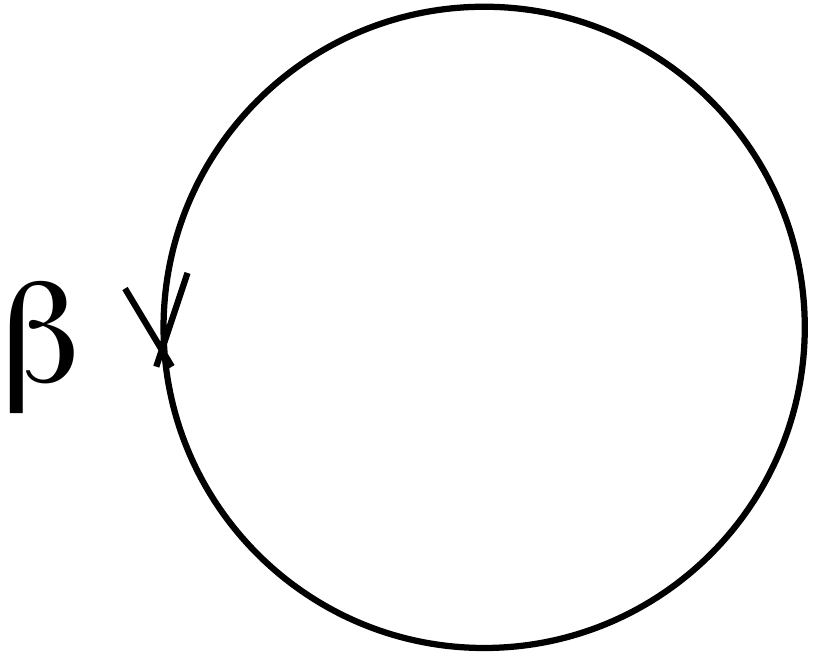}}  \]
 After applying these isomorphisms to all webs, each canonical resolution $\Gamma_\phi$ is isomorphic---in $\textbf{Foam}_{/\ell}$---to a collection of oriented circles labeled $\alpha$ or $\beta,$ some of which may be concentric. To each such  circle $C$ we assign a mod 2 invariant in  the following way:
 \begin{enumerate}
 \item[(i)] take the mod $2$ number of circles which separate $C$ from infinity; in other words, draw a ray in the plane from $C$ to infinity and take the number of times it intersects the other circles, mod $2.$ 
 \item[(ii)] To this number, add $1$ if it has the clockwise orientation, and $0$ if it has the counterclockwise orientation. 
 \item[(iii)] Moreover, add $1$ if $C$ is labeled $\beta,$ and $0$ if it is labeled $\alpha.$ 
 \end{enumerate}
 
 Finally, label the circle $C$ by $z_1$ if the resulting invariant is $0$ mod $2,$ and by $z_2$ if it is $1$ mod $2.$ An example is given below.
  
   \[ \raisebox{-33pt}{\includegraphics[height=1.3in]{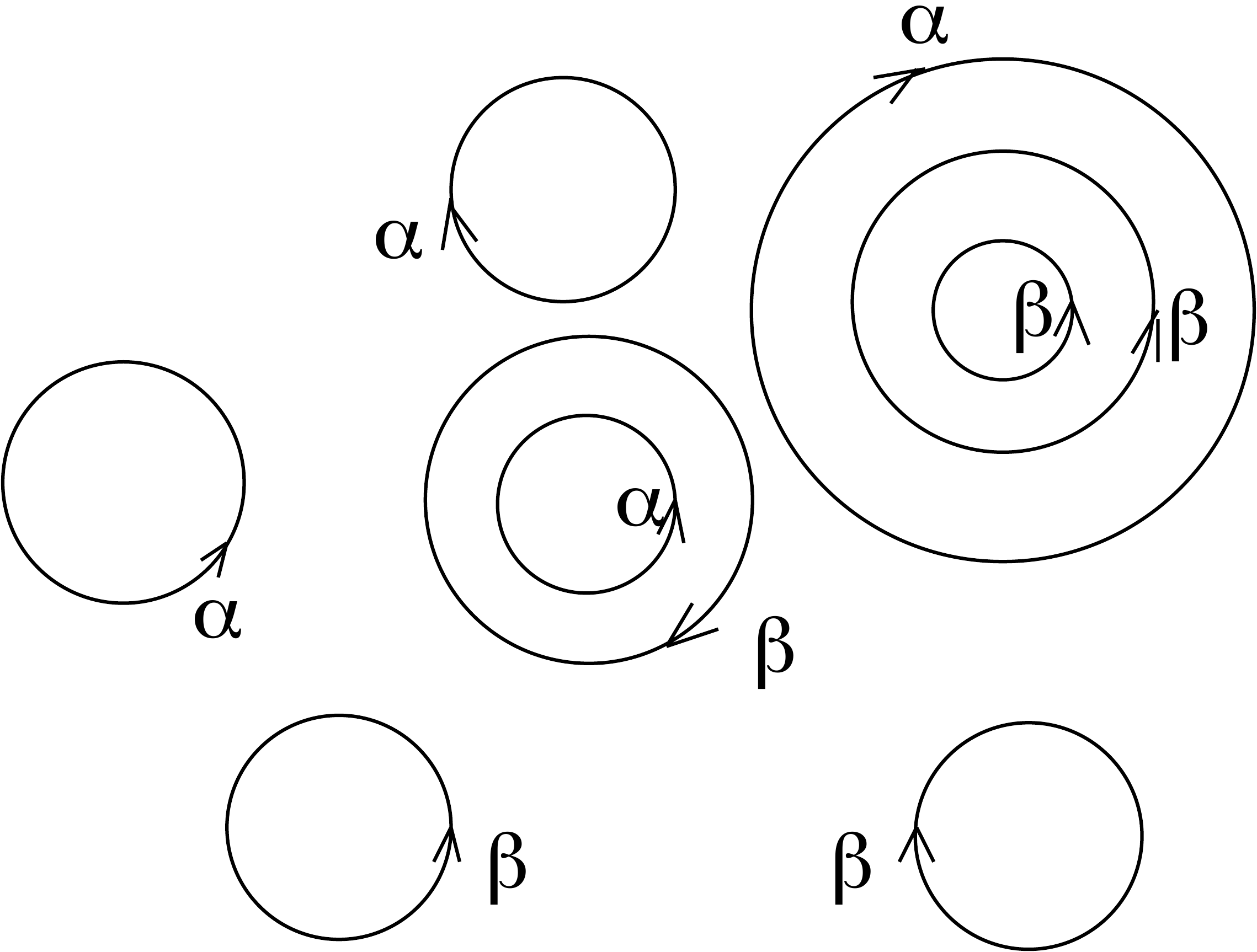}} \qquad \longrightarrow \qquad \raisebox{-33pt}{\includegraphics[height=1.3in]{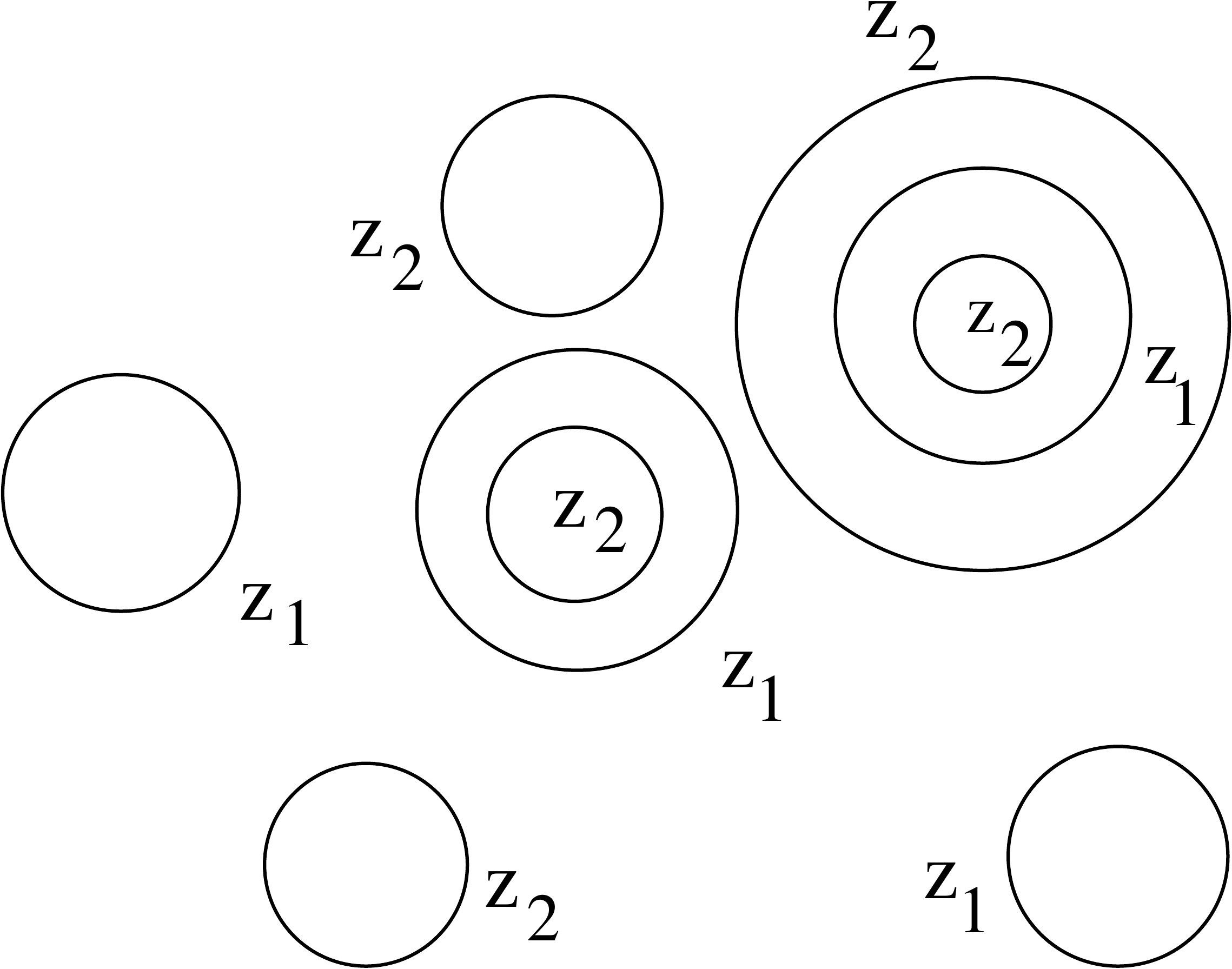}} \]
 
This procedure allows us to regard each canonical generator $h_{\phi} \in \H(L)$ as a tensor product of the idempotent elements $z_1$ and $z_2,$ making the proofs of the following two statements more manageable, via the TQFT associated to $\calA.$ 

\begin{proposition}\label{prop: Reidemeister moves}
Suppose that $C$ is an elementary link cobordism from $D_0$ to $D_1,$ where diagrams $D_0$ and $D_1$ are related by a Reidemeister move and that $\phi_0$ and $\phi_1$ are canonical states of $D_0$ and $D_1$ compatible with each other. Then the isomorphism between the homologies of the two diagrams---induced by the Reidemeister move---maps any canonical generator to a nonzero multiple of the compatible canonical generator. That is, $ \mathcal{L}_C (h_{\phi_0}) = \lambda h_{\phi_1},$ for some $\lambda \in \bbC ^*.$ 
\end{proposition}

\begin{proof} \textit{Reidemeister move I}. Consider the diagrams $D_0=\raisebox{-13pt}{\includegraphics[height=0.4in]{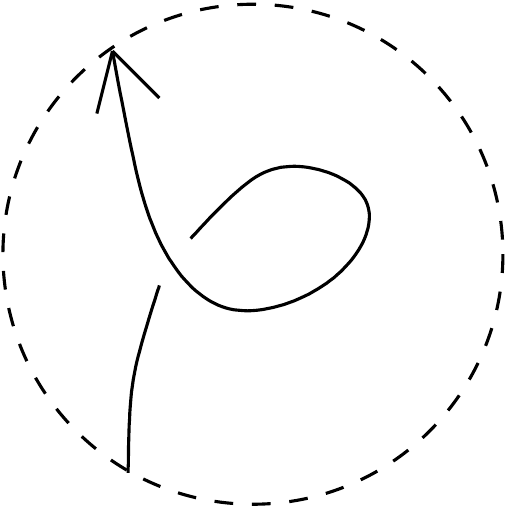}}$ and $D_1=\raisebox{-13pt}{\includegraphics[height=0.4in]{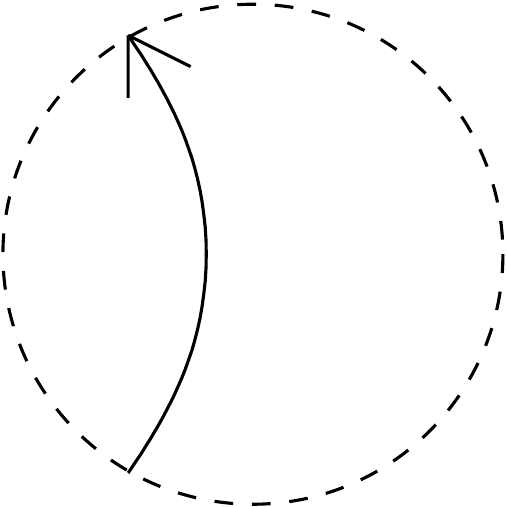}}$.
The associated formal complexes are depicted below, where we underlined the objects at the cohomological degree zero.
\[[D_0] = (0 \longrightarrow \underline{\raisebox{-8pt}{\includegraphics[height=0.3in]{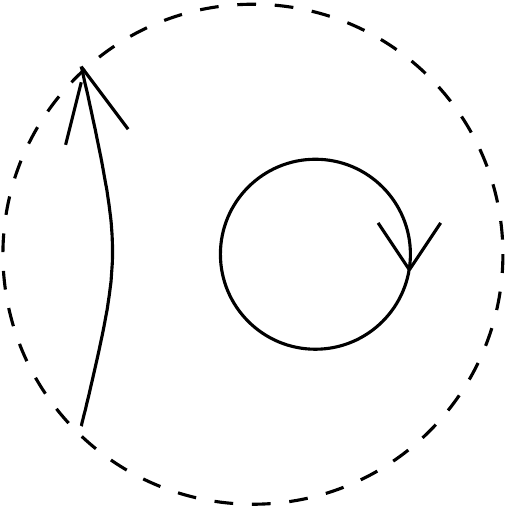}}\{-1\}}\stackrel{\raisebox{-13pt}{\includegraphics[height=0.3in]{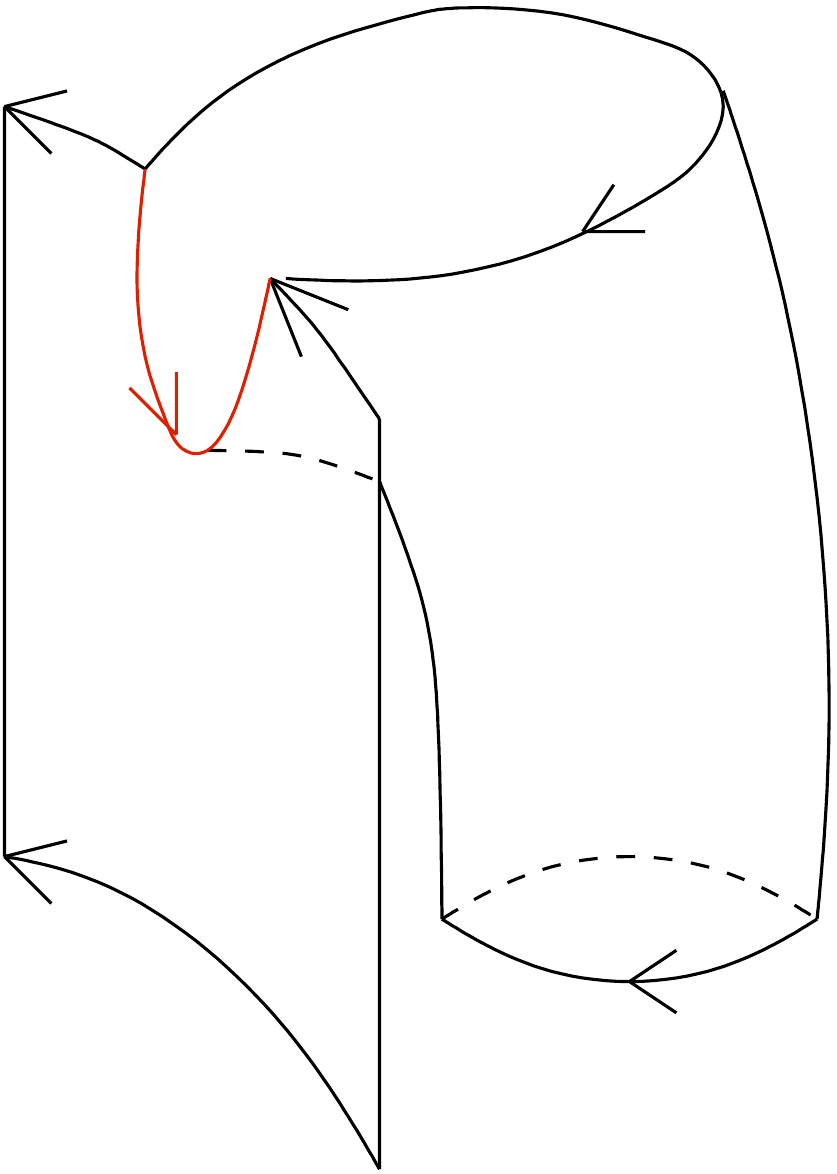}}}{\longrightarrow}\raisebox{-8pt} {\includegraphics[height=0.3in]{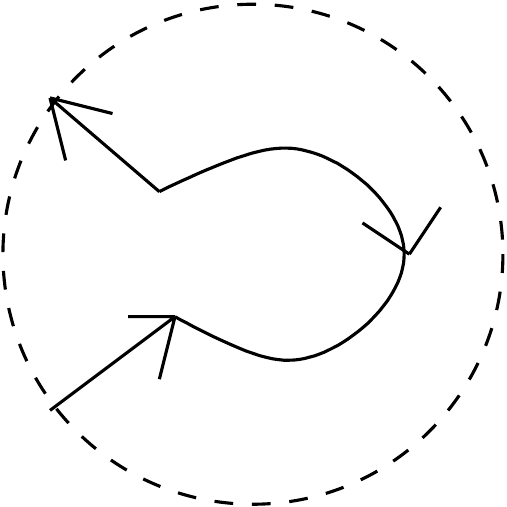}}\{-2\} \longrightarrow 0)\quad \text{and} \quad
[D_1]  = (0 \longrightarrow \underline{\raisebox{-8pt} {\includegraphics[height=0.3in]{reid1-1.pdf}}} \longrightarrow 0)\]

In~\cite{CC1} we constructed chain maps $g \co [D_0] \to [D_1]$ and  $f \co [D_1] \to [D_0]$ defining a homotopy equivalence, as follows
\[g^0 =  \raisebox{-10pt}{\includegraphics[height=0.4in]{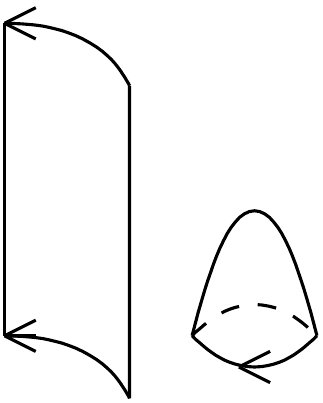}}\,, \quad f^0 = \raisebox{-10pt}{\includegraphics[height=0.4in]{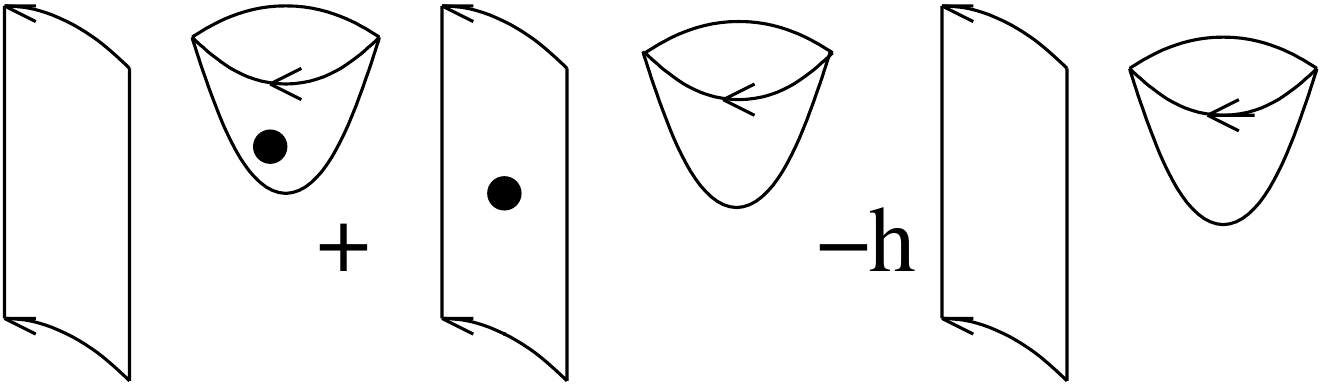}}\,, \quad g^1 = f^1 = 0.\]

The visible strand of $D_0$ and $D_1$ can be labeled by either $\alpha$ or $\beta$; assume the latter (the proof for the other case follows similarly). The corresponding canonical resolution of $D_0$  is sent via the (labeled) map $g^0$ to the corresponding compatible canonical resolution of $D_1,$ as shown below: 
\begin{figure}[ht!]
$\xymatrix@C=25mm@R = 20mm{
\raisebox{-8pt}{\includegraphics[height=0.4in]{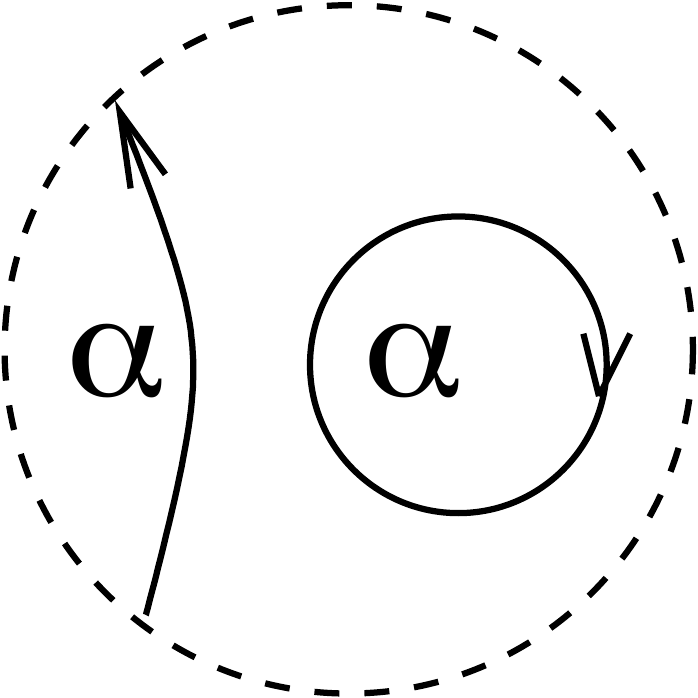}} \ar@<3pt>[r]^ {g^0\; = \; {\raisebox{-12pt}{\includegraphics[height=0.4in]{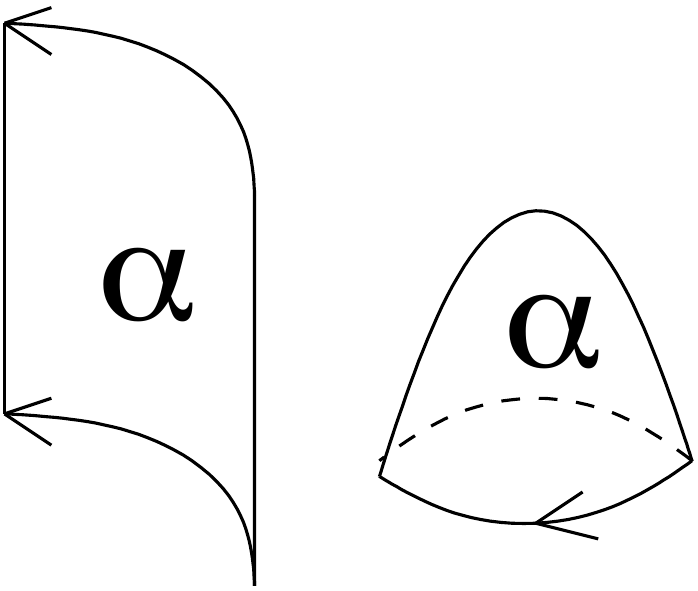}}}} & \raisebox{-8pt}{\includegraphics[height=0.4in]{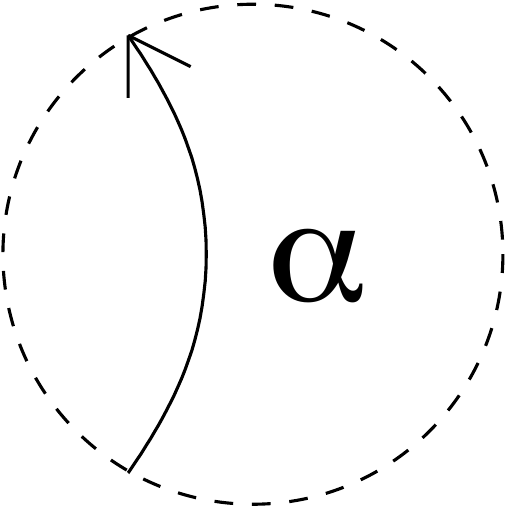}}
}$
\end{figure}

Therefore $\mathcal{L}_C (h_{\phi_0}) = \lambda h_{\phi_1},$ for some $\lambda \in \bbC^*.$ The case when the two diagrams differ by a positive kink is very similar, thus we omit it.

\textit{Reidemester IIa}. Consider $D_0 = \raisebox{-13pt}{\includegraphics[height=0.4in]{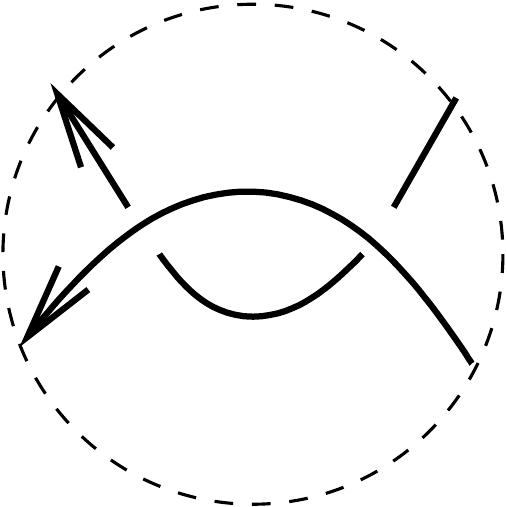}}$ and $D_1 = \raisebox{-13pt}{\includegraphics[height=0.4in]{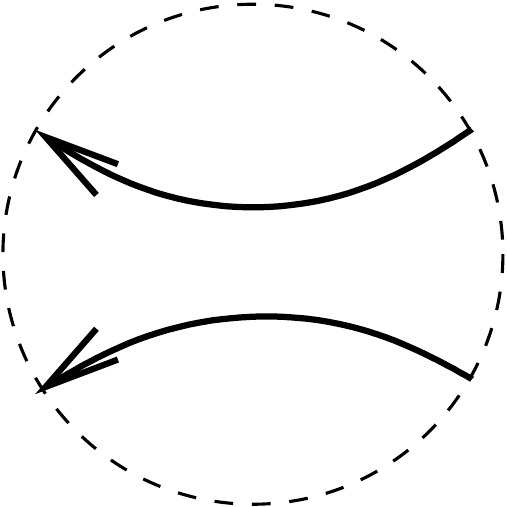}}$, and their associated chain complexes
\[[D_0] :(0 \longrightarrow \raisebox{-8pt} {\includegraphics[height=0.3in]{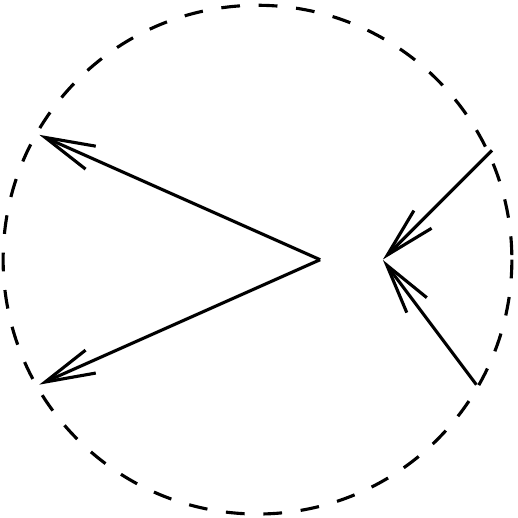}}\{1\} \longrightarrow \underline{\raisebox{-8pt} {\includegraphics[height=0.3in]{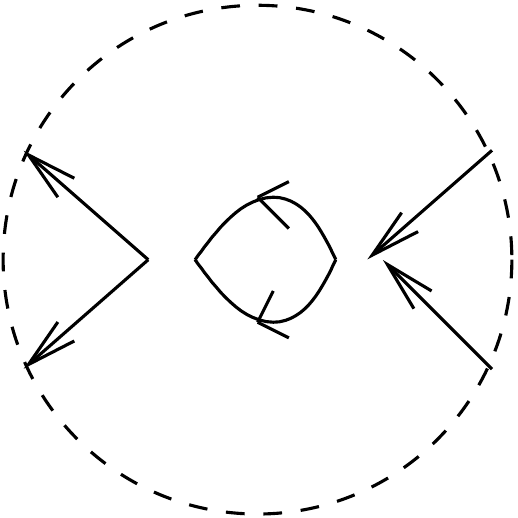}} \oplus \raisebox{-8pt} {\includegraphics[height=0.3in]{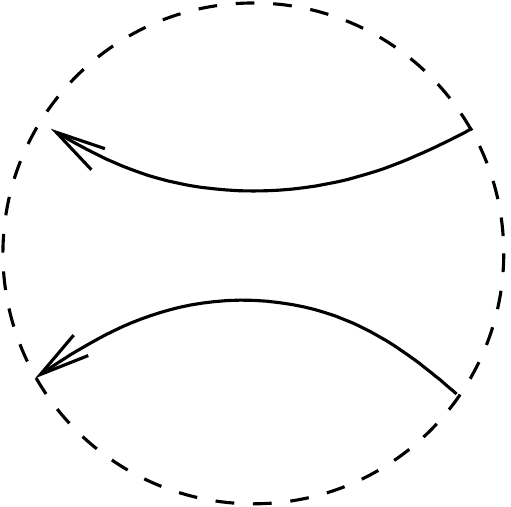}}}\longrightarrow \raisebox{-8pt} {\includegraphics[height=0.3in]{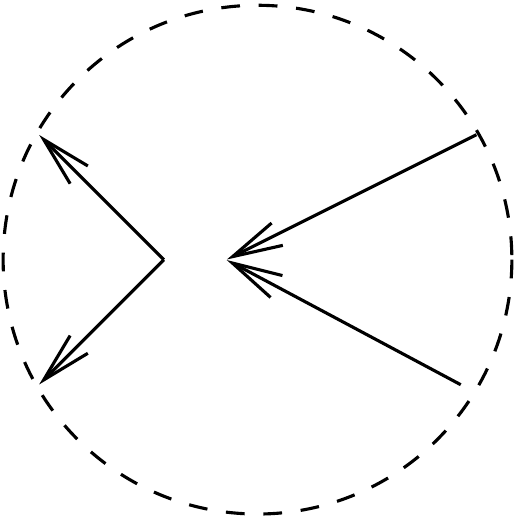}}\{-1\} )\, \text {and}\] \[[D_1]: (0 \longrightarrow \underline{\raisebox{-8pt} {\includegraphics[height=0.3in]{2arcslo.pdf}}} \longrightarrow 0).\] 
The chain maps $g \co [D_0] \to [D_1]$ and  $f \co [D_1] \to [D_0]$ defining a homotopy equivalence between $[D_0]$ and $[D_1]$ are given by
\[ g^{-1} = f^{-1} = 0, \quad g^0 = \left (\begin{array}{cc }-\raisebox{-15pt}{\includegraphics[height=0.5in]{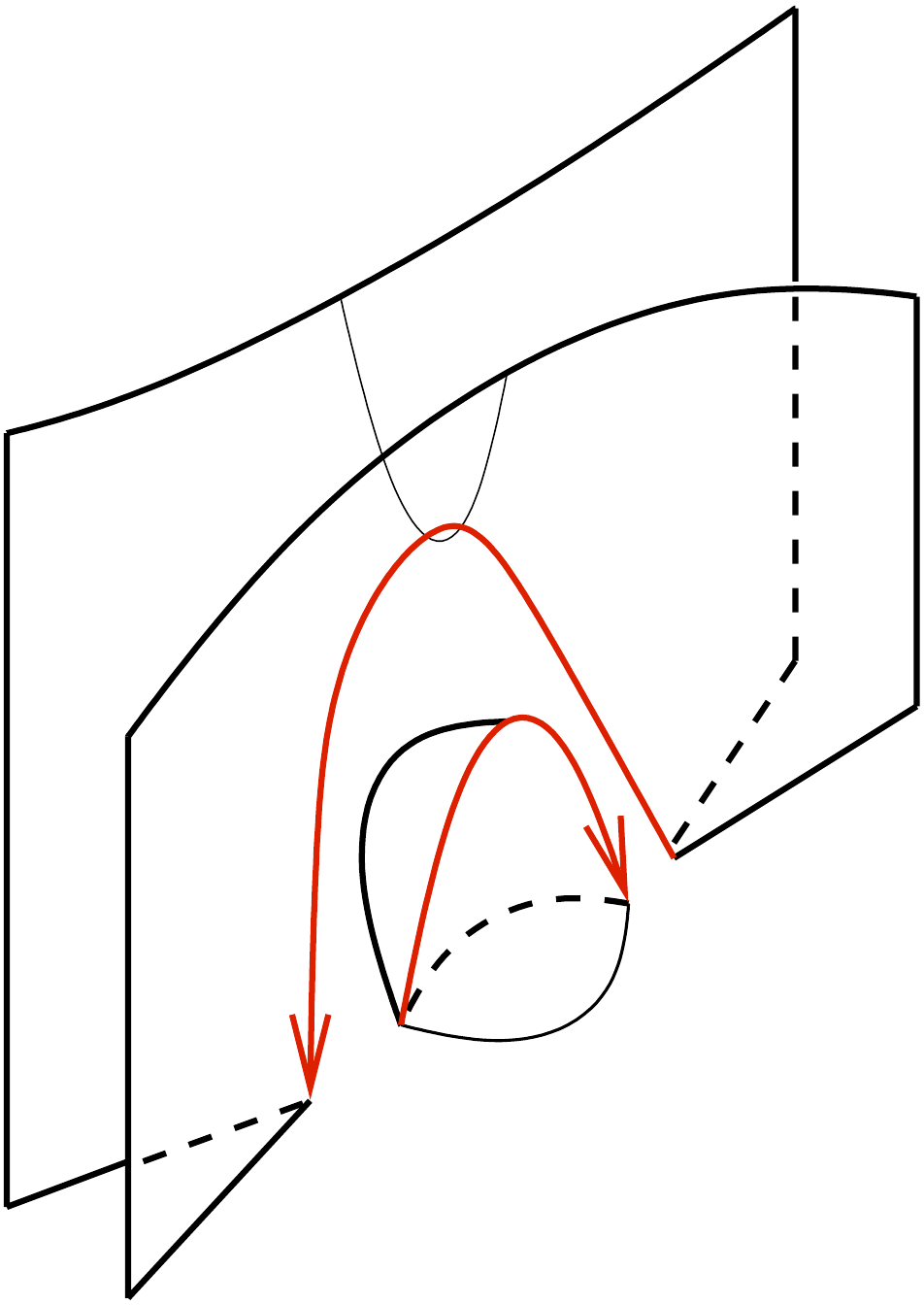}}\,, &\id \end{array} \right), \, \quad f^0 = \left( \begin{array}{c} \raisebox{-10pt}{\includegraphics[height=0.5in]{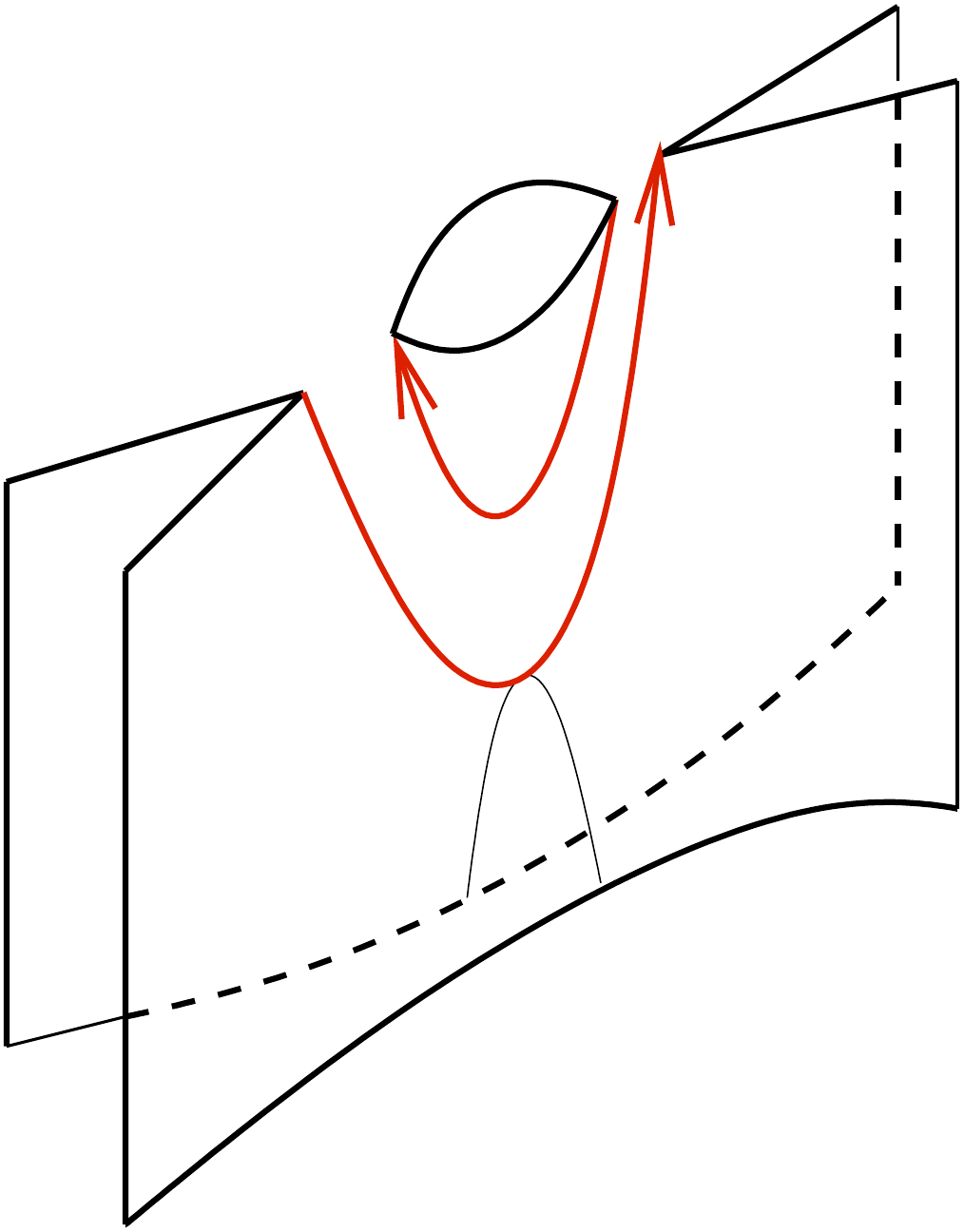}}\\ \id \end{array} \right), \, \quad g^1 = f^1 = 0.  \]

\textit{Case 1}. If the two visible strands of the diagrams $D_0$ and $D_1$ belong to the same component of the link, then they have the same label, say $\alpha$:
$$D_0=\raisebox{-13pt}{\includegraphics[height=0.4in]{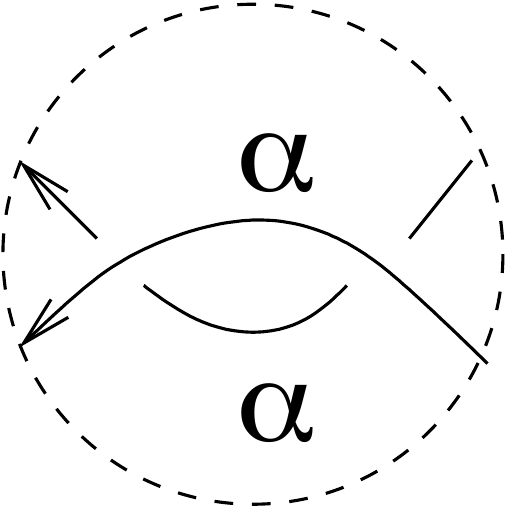}}\qquad
D_1=\raisebox{-13pt}{\includegraphics[height=0.4in]{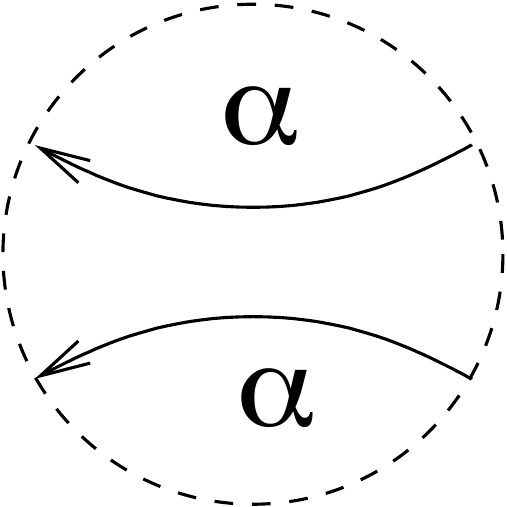}}\ $$
The canonical resolution of both labeled diagrams is $\raisebox{-13pt}{\includegraphics[height=0.4in]{twoarcs-alpha.pdf}}\ $, and the map between them constructed in the proof of invariance under the Reidemeister  2a move is the identity map, i.e., two ``curtains" labeled by $\alpha.$ Therefore $h_{\phi_0}\stackrel{\calL_C} {\longrightarrow} h_{\phi_1}.$

\textit{Case 2}. If the two visible strands of $D_0$ and $D_1$ belong to different components of the link, then they might have the same labels (and we are in the same situation as in the first case) or different labels. Assume that the strands are labeled as shown below (notice that these are compatible states for the two diagrams): 
$$D_0=\raisebox{-13pt}{\includegraphics[height=0.4in]{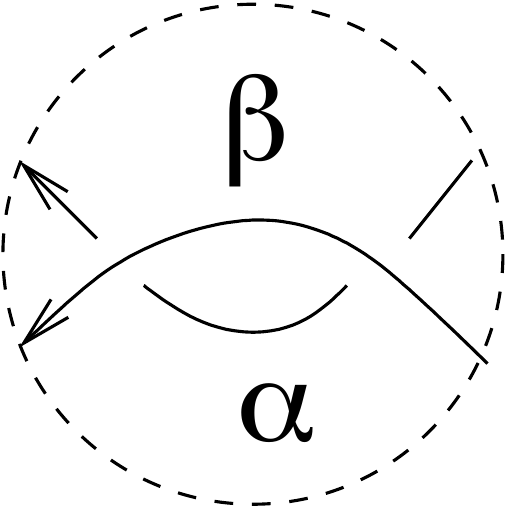}}\qquad
D_1=\raisebox{-13pt}{\includegraphics[height=0.4in]{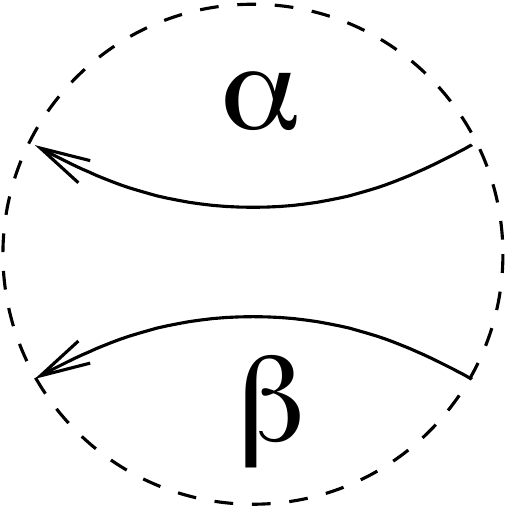}}\ $$
This time, the canonical resolution of this state for $D_0$ is $\raisebox{-10pt}{\includegraphics[height=0.4in]{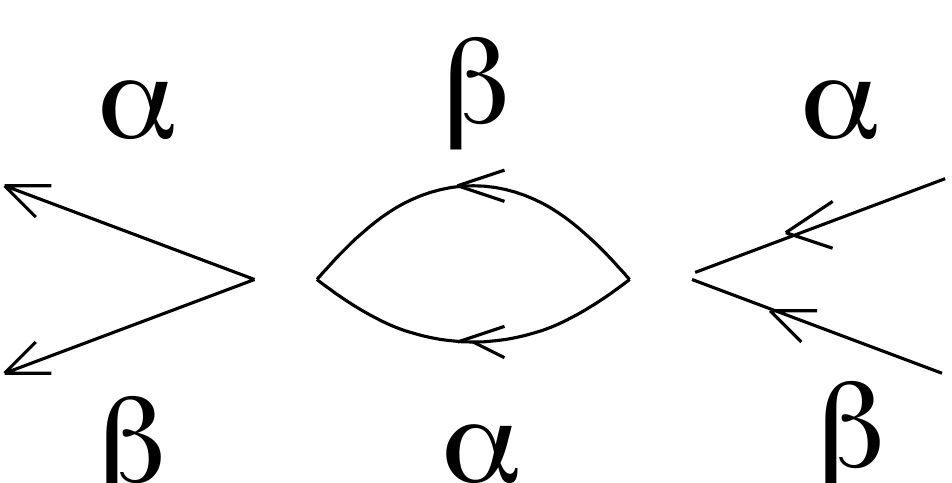}},$ which is mapped to (a nonzero multiple of) the canonical resolution of $D_1$ via the map:
\begin{figure}[ht!]
$\xymatrix@C=20mm@R = 25mm{
\raisebox{-8pt}{\includegraphics[height=0.4in]{reid2a-2-alphabeta.pdf}} \ar@<2pt>[r]^ {\quad - \,\raisebox{-18pt}{\includegraphics[height=0.65in]{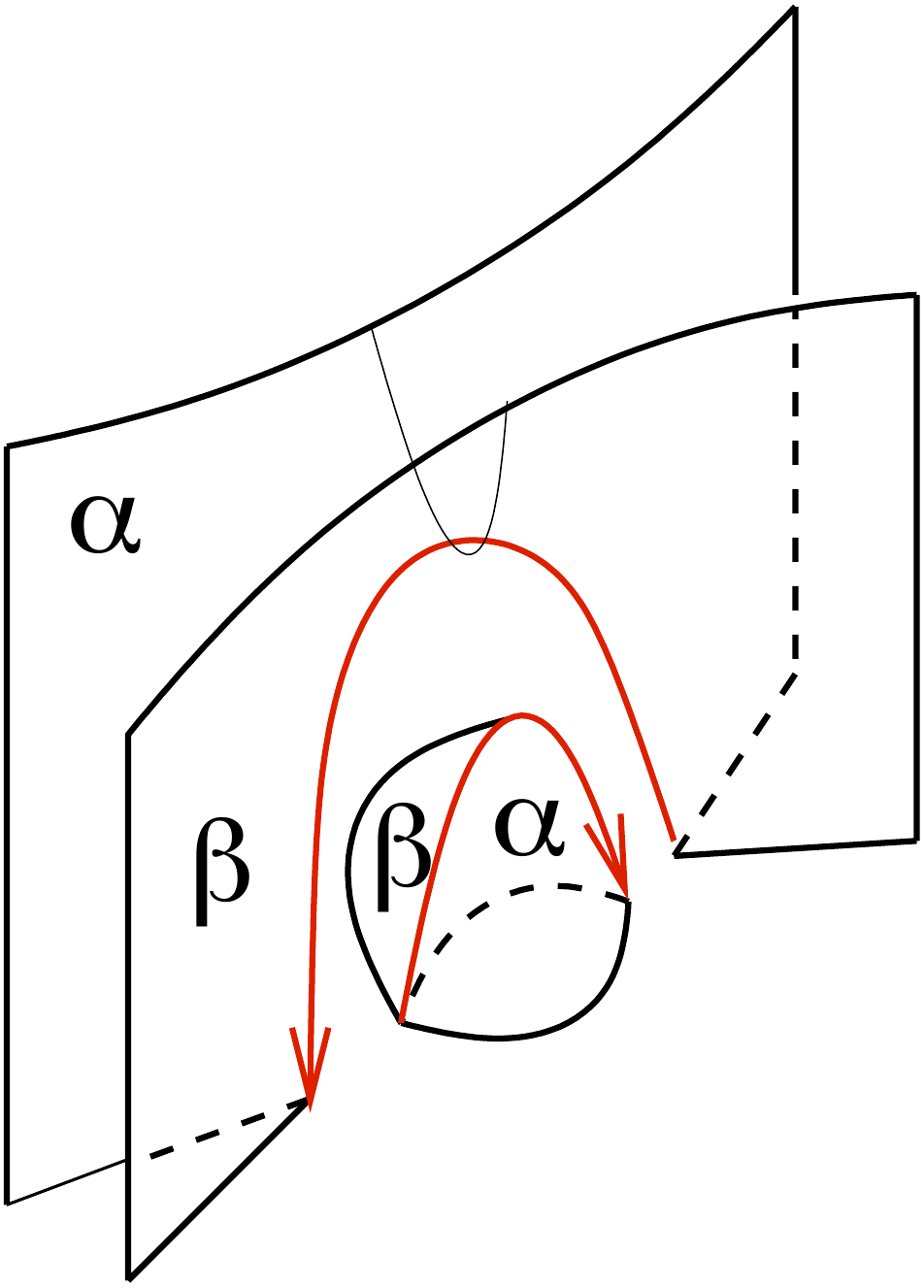}}} & \raisebox{-5pt}{\includegraphics[height=0.3in]{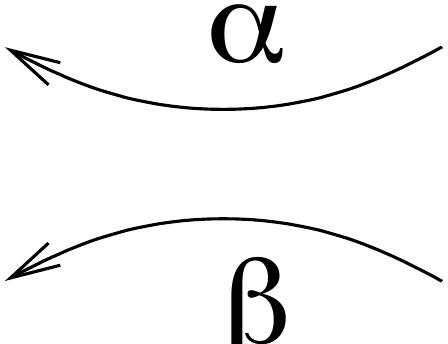}}}$
\end{figure}

Therefore the canonical generator $h_{\phi_0}$ is sent via the map $\calL_C$ to a nonzero multiple of the compatible canonical generator $h_{\phi_1}.$

\textit{Reidemester IIb}. Consider the diagrams $D_0 = \raisebox{-13pt}{\includegraphics[height=0.4in]{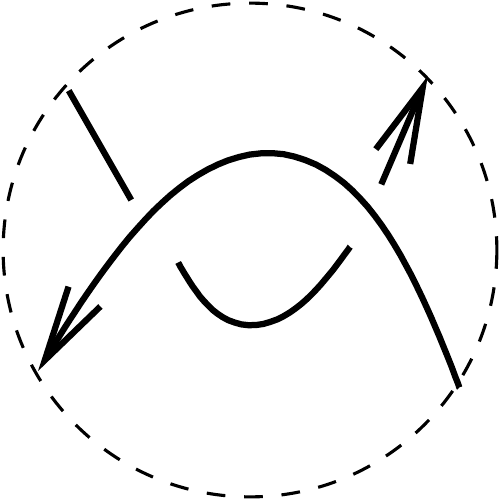}}$ and $D_1=\raisebox{-13pt}{\includegraphics[height=0.4in]{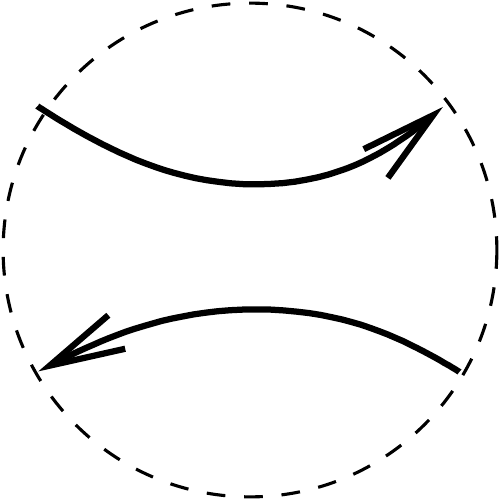}}$ and their associated complexes
 \[[D_0] :(0 \longrightarrow \raisebox{-8pt} {\includegraphics[height=0.3in]{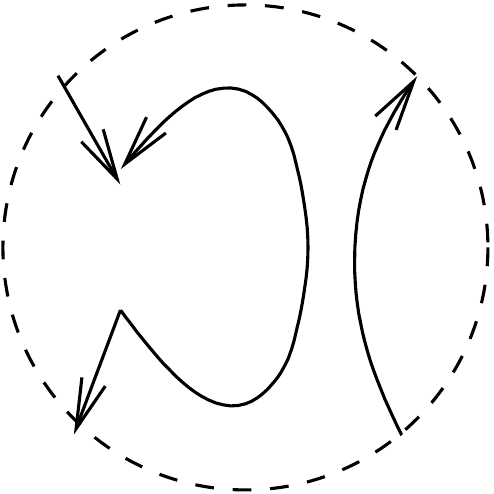}}\{1\} \longrightarrow \underline { \raisebox{-8pt} {\includegraphics[height=0.3in]{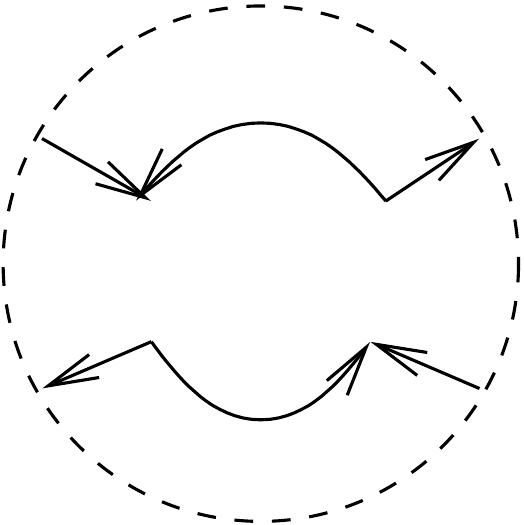}}\oplus \raisebox{-8pt} {\includegraphics[height=0.3in]{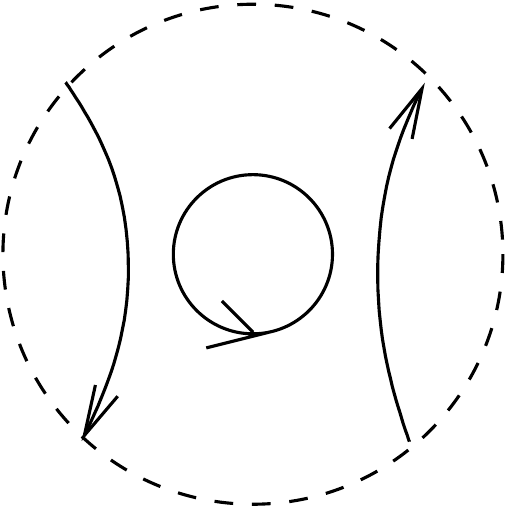}}} \longrightarrow \raisebox{-8pt} {\includegraphics[height=0.3in]{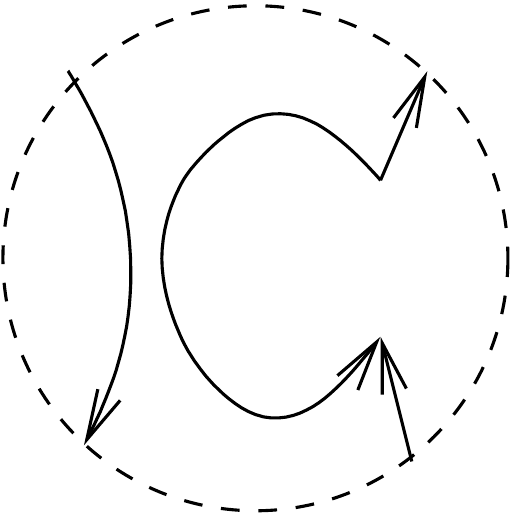}}\{-1\} )\, \text {and}\] \[[D_1]: (0 \longrightarrow \underline{\raisebox{-8pt} {\includegraphics[height=0.3in]{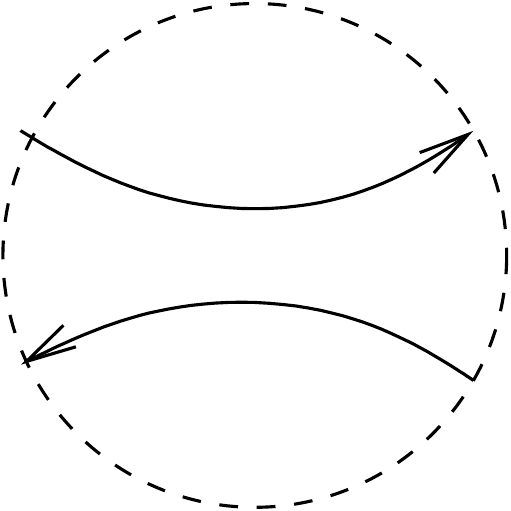}}} \longrightarrow 0).\]  
The chain maps $g \co [D_0] \to [D_1]$ and $f \co [D_1] \to [D_0]$ defining a homotopy equivalence between $[D_0]$ and $[D_1]$ are given by
\[ g^{-1} = f^{-1} = 0, \, g^0 = \left (\begin{array}{cc } -\,\raisebox{-10pt}{\includegraphics[height=0.4in]{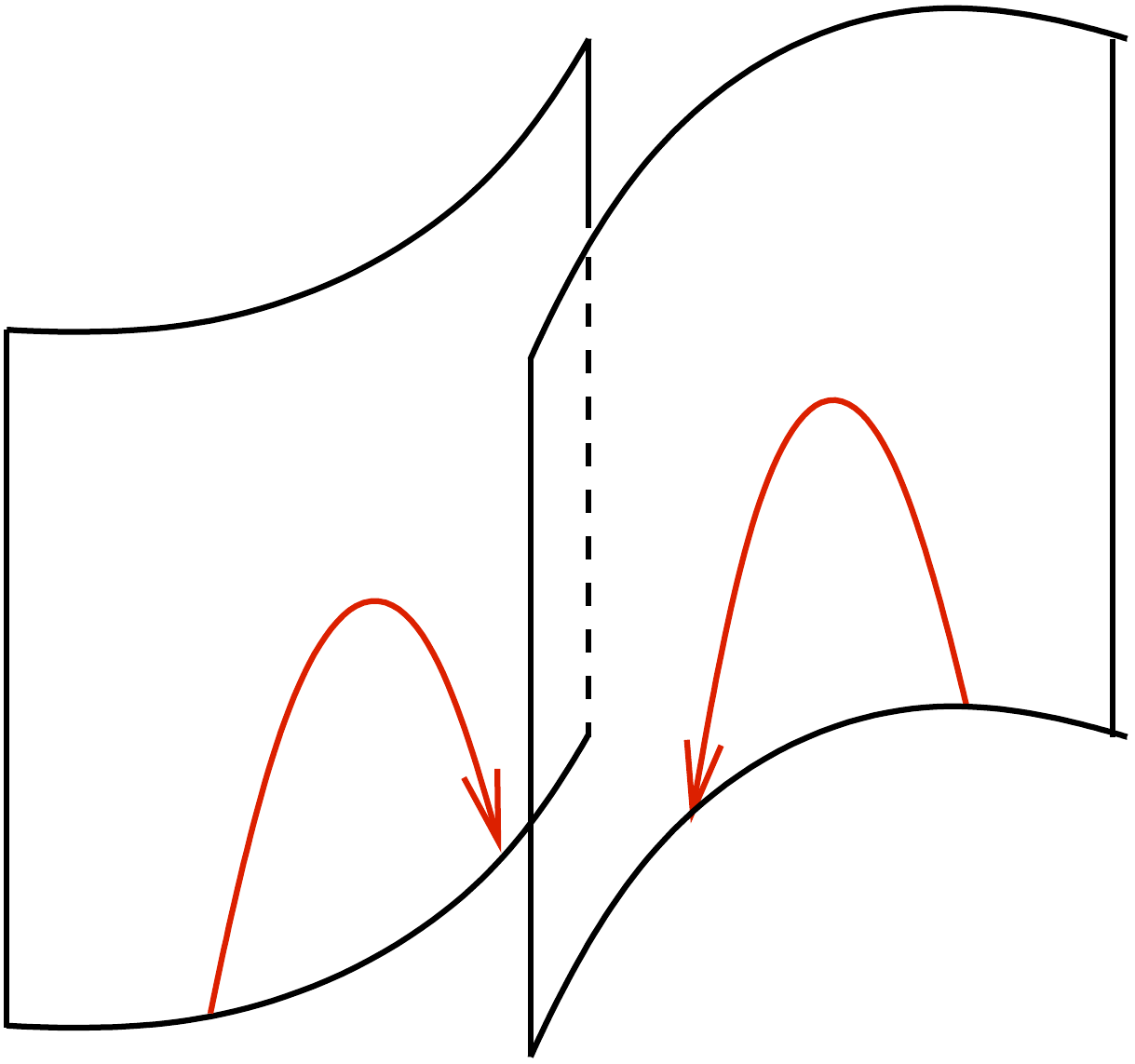}}\,, &\raisebox{-10pt}{\includegraphics[height=0.5in]{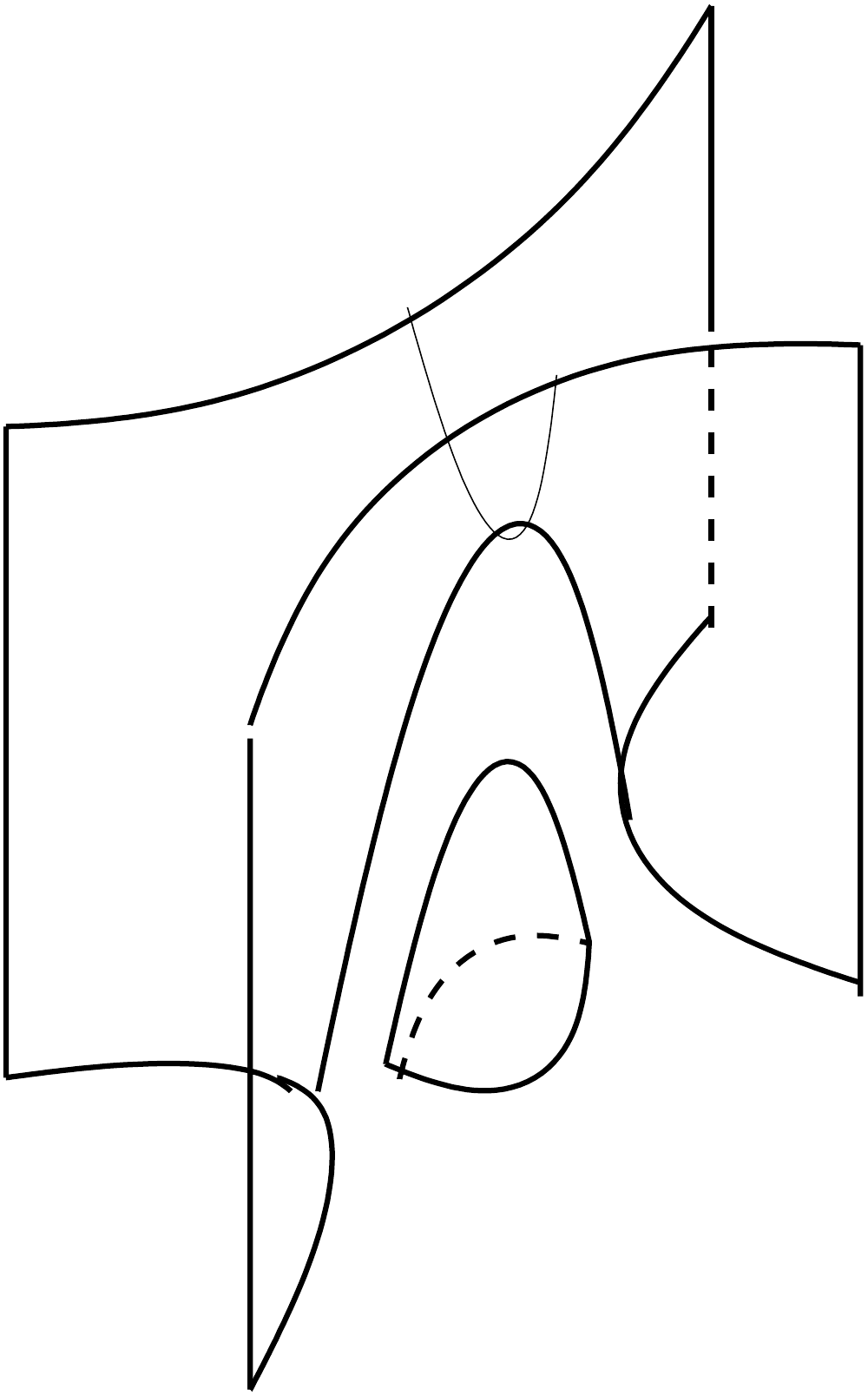}} \end{array} \right), \, f^0 = \left( \begin{array}{c} \raisebox{-10pt}{\includegraphics[height=0.4in]{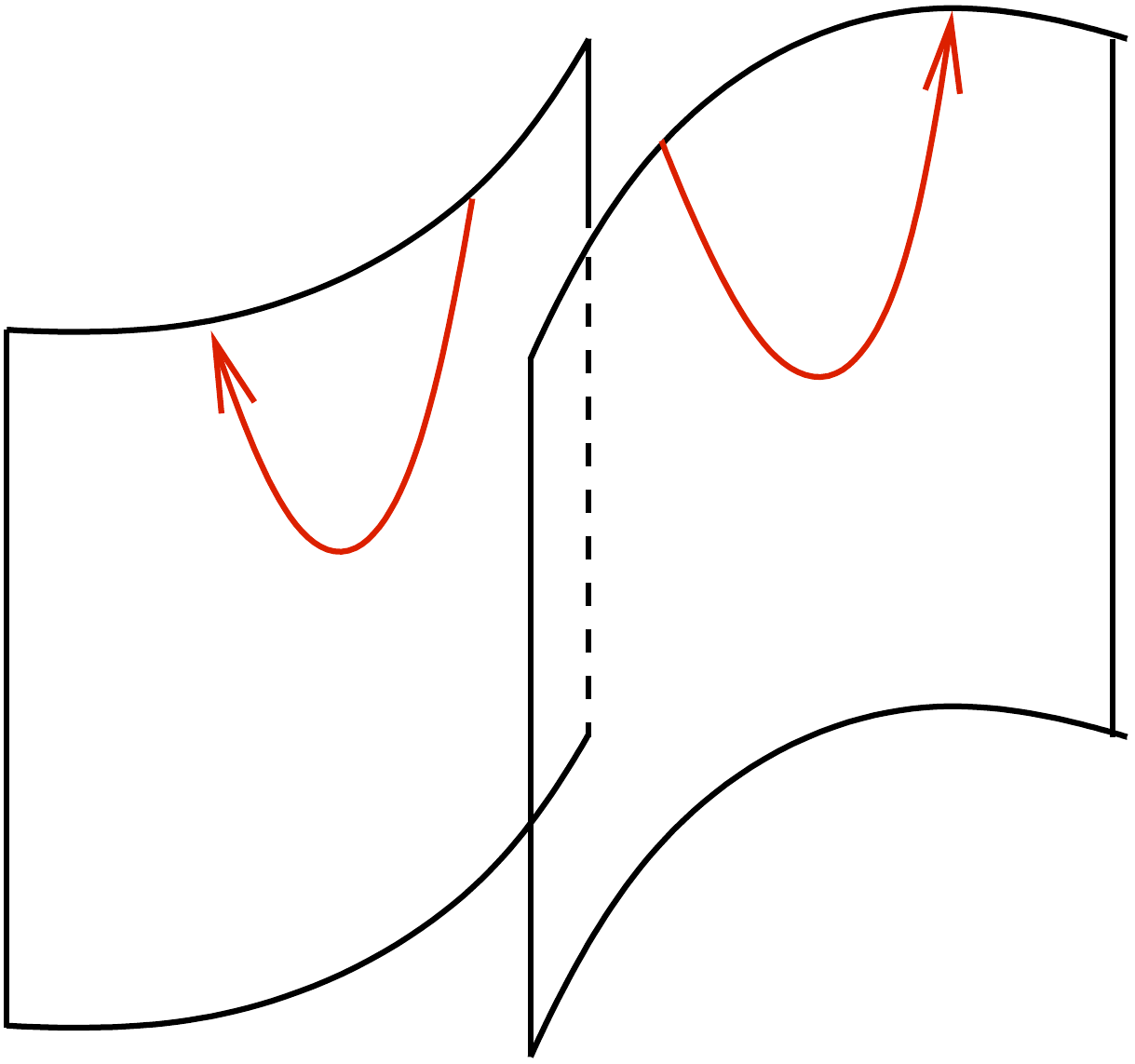}}\\  \raisebox{-10pt}{\includegraphics[height=0.5in]{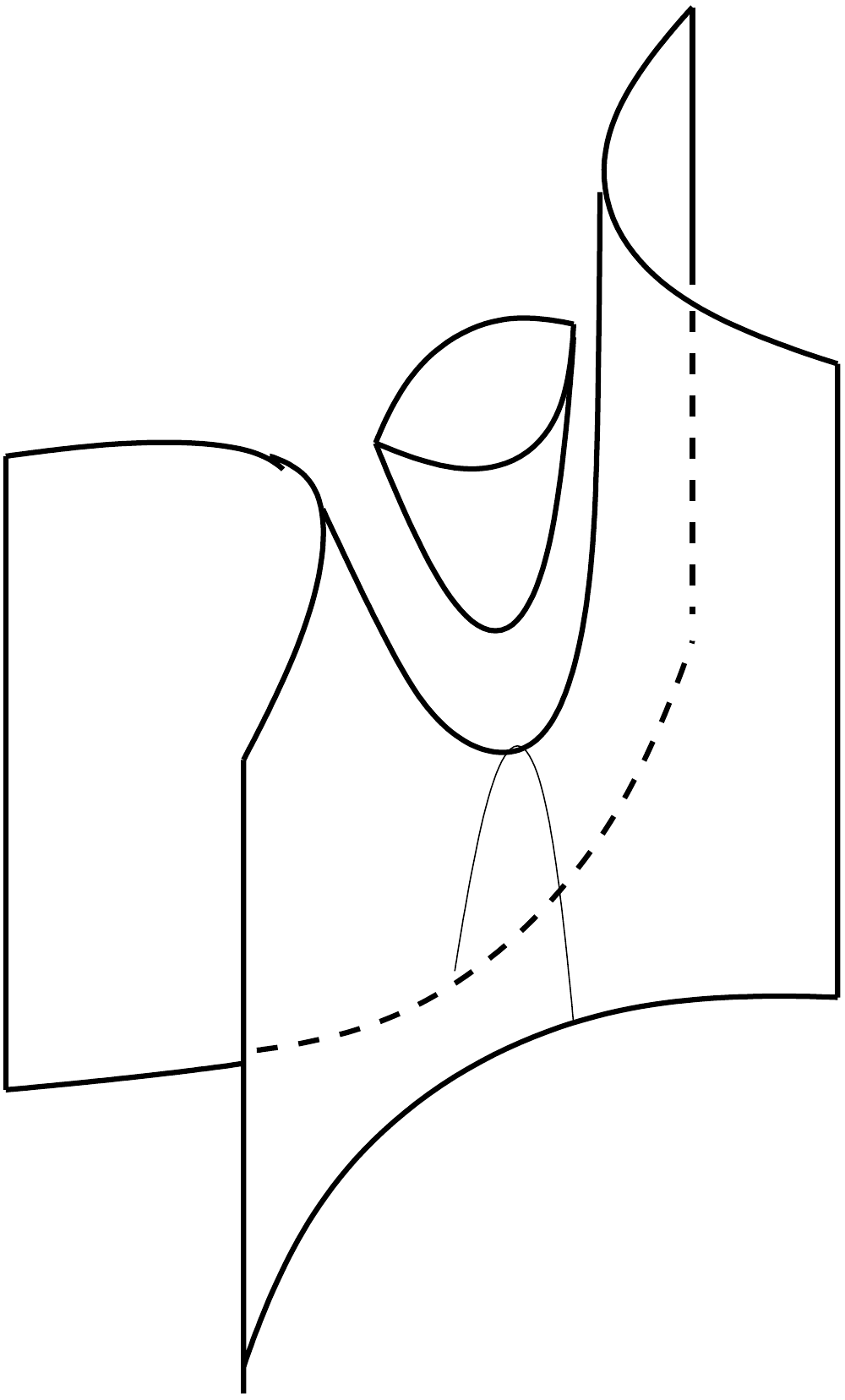}} \end{array} \right), \, g^1 = f^1 = 0. \]
There are again two cases to consider.

\textit{Case 1}. Assume that the visible two strands of the diagrams $D_0$ and $D_1$ have the same  label, say $\beta.$ Then the canonical resolutions of $D_0$ is $\,\raisebox{-8pt} {\includegraphics[height=0.3in]{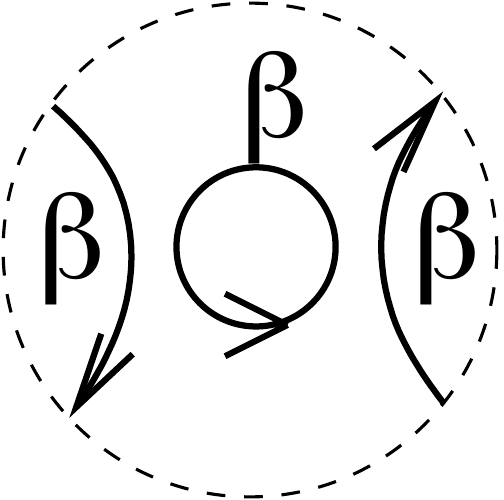}}\,$ and is mapped to the canonical resolution of $D_1$ via the map:
 \begin{figure}[ht!]
$\xymatrix@C=25mm@R = 25mm{
\raisebox{-8pt}{\includegraphics[height=0.4in]{reid2b-2-beta.pdf}} \ar@<2pt>[r]^ {\quad  \,\raisebox{-13pt}{\includegraphics[height=0.7in]{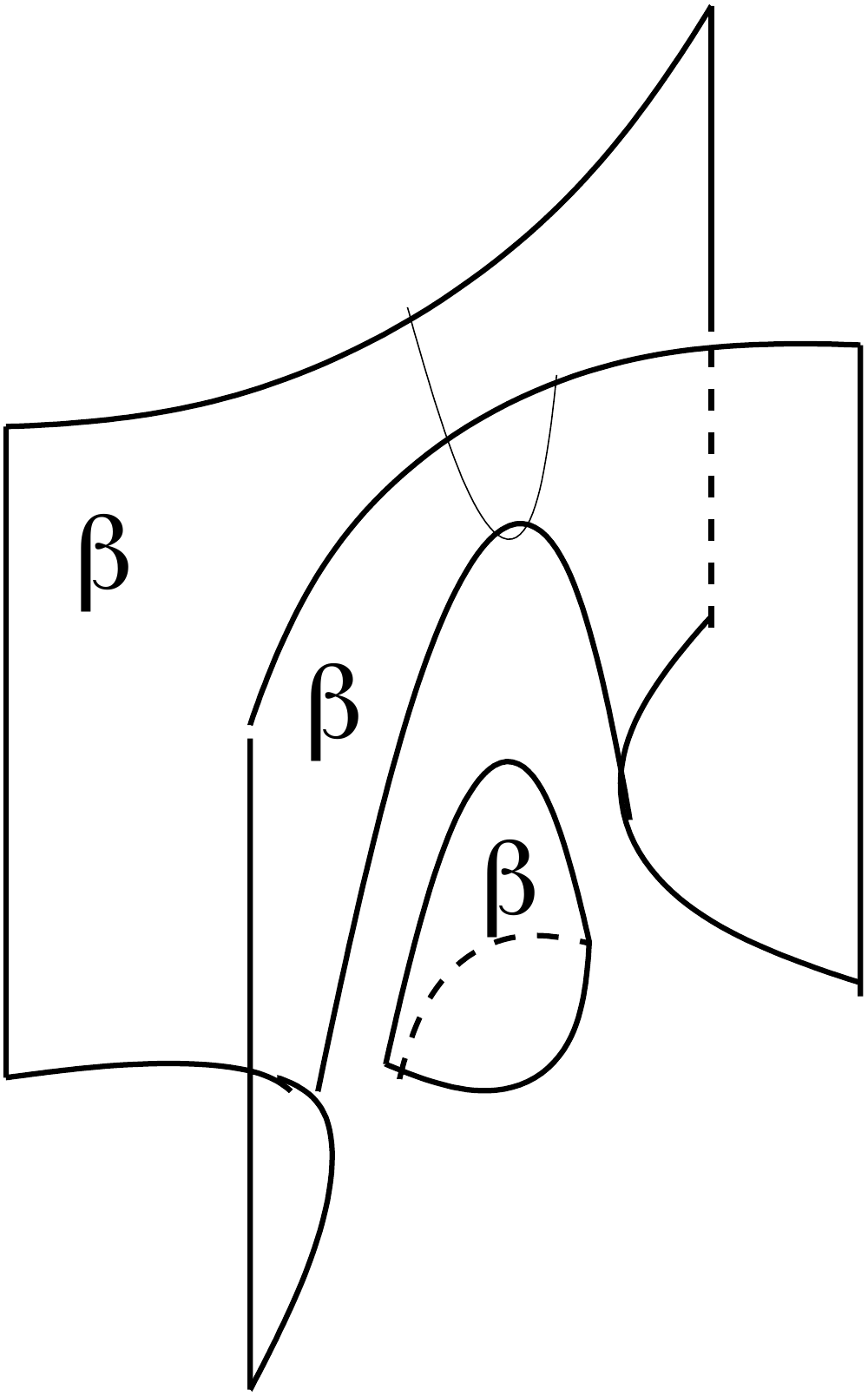}}} & \raisebox{-5pt}{\includegraphics[height=0.4in]{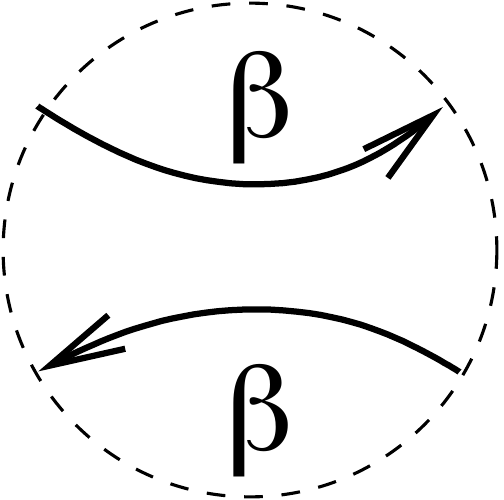}}}$
\end{figure}

\textit{Case 2}. Now we assume the visible two strands have different labels, say as shown below.
$$D_0=\raisebox{-13pt}{\includegraphics[height=0.4in]{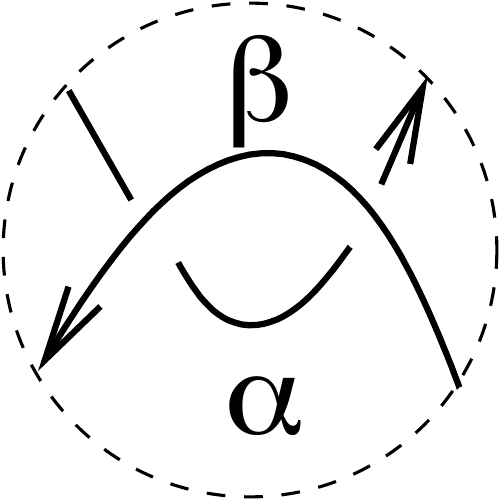}}\qquad
D_1=\raisebox{-13pt}{\includegraphics[height=0.4in]{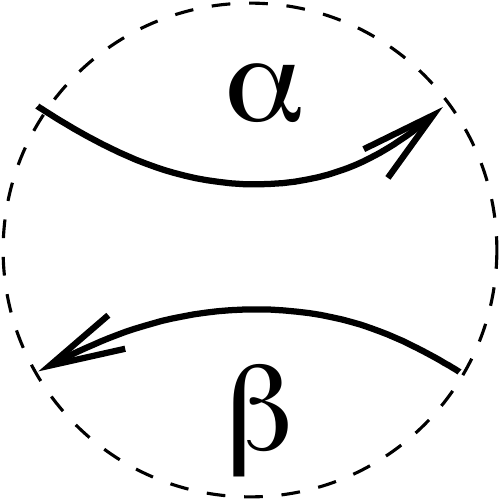}}\ $$
The canonical resolution of $D_0$ is $\,\raisebox{-8pt} {\includegraphics[height=0.35in]{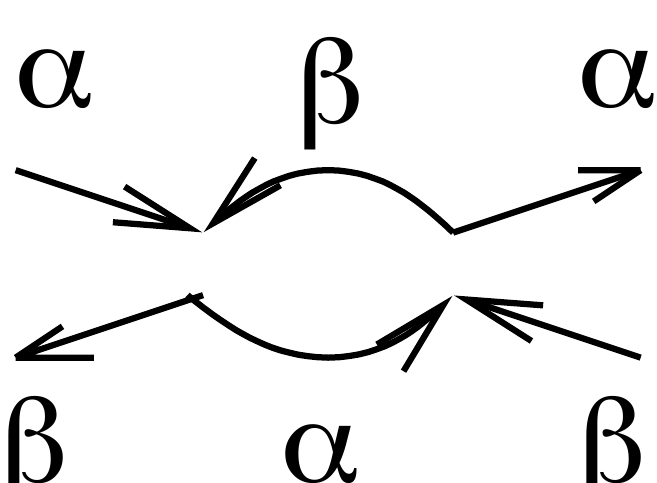}}\,$ and is mapped to a nonzero multiple of the canonical resolution of $D_1$ via the map:
 \begin{figure}[ht!]
$\xymatrix@C=25mm@R = 25mm{
\raisebox{-8pt}{\includegraphics[height=0.4in]{reid2b-3-alphabeta.pdf}} \ar@<2pt>[r]^ {\quad  \,-\,\raisebox{-18pt}{\includegraphics[height=0.6in]{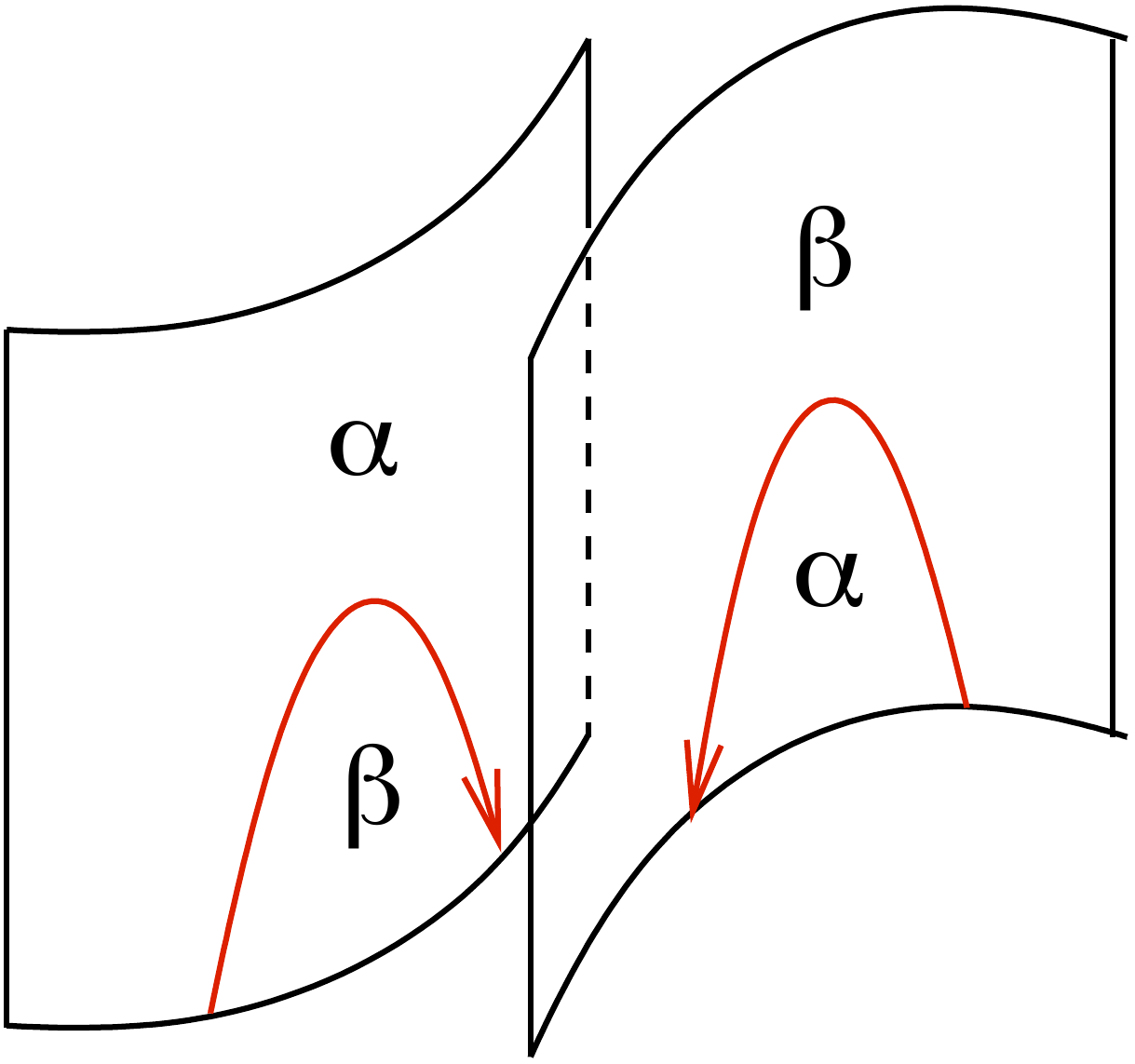}}} & \raisebox{-5pt}{\includegraphics[height=0.35in]{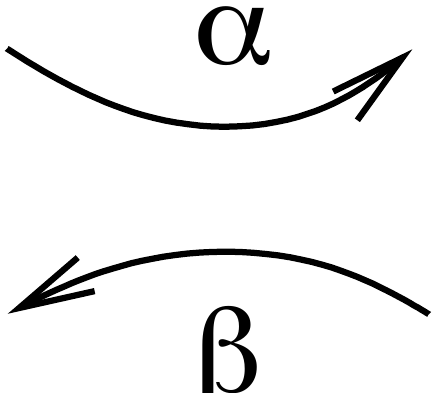}}}$
\end{figure}

In both cases, $\calL_C(h_{\phi_0}) = \lambda h_{\phi_1},$ for some $\lambda \in \mathbb{C}^*.$

 \textit{Reidemeister move III}. Consider the diagrams $D_0=\raisebox{-11pt}{\includegraphics[height=0.4in]{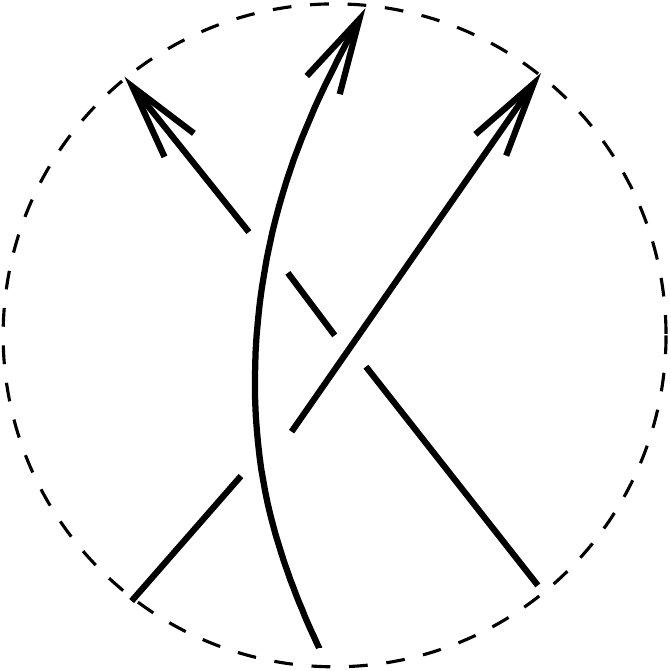}}
 $ and $D_1=\raisebox{-10pt}{\includegraphics[height=0.4in]{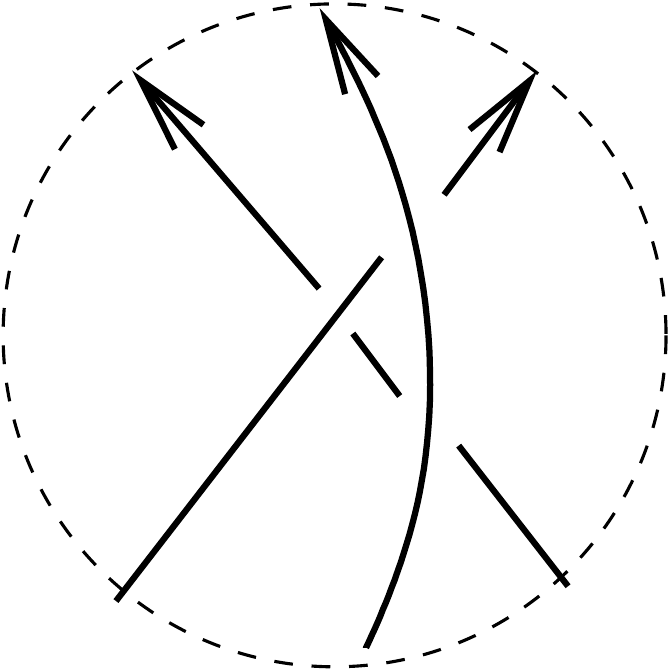}}$.  
Up to permutations (that is, up to replacing $\alpha$ by $\beta$ and vice versa), there are four cases to consider for the possible labelings of the involved arcs.

\textit{Case 1}.  \[D_0=\raisebox{-13pt}{\includegraphics[height=0.45in]{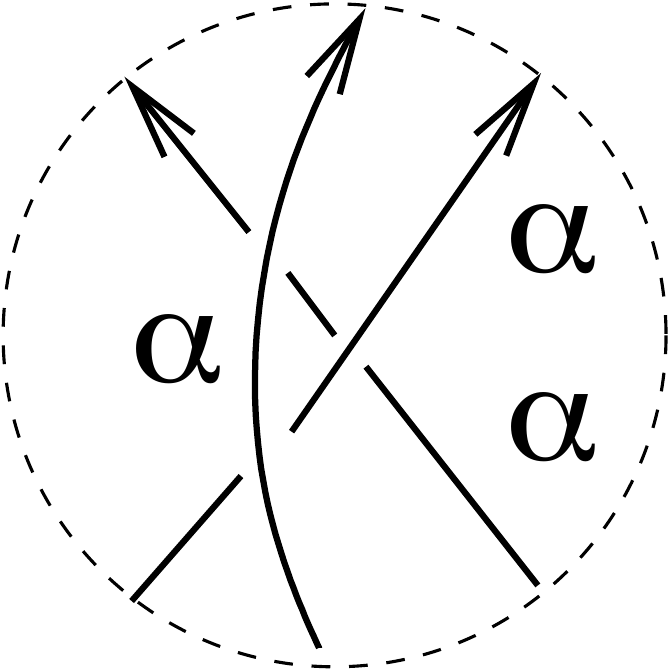}}\qquad
D_1=\raisebox{-13pt}{\includegraphics[height=0.45in]{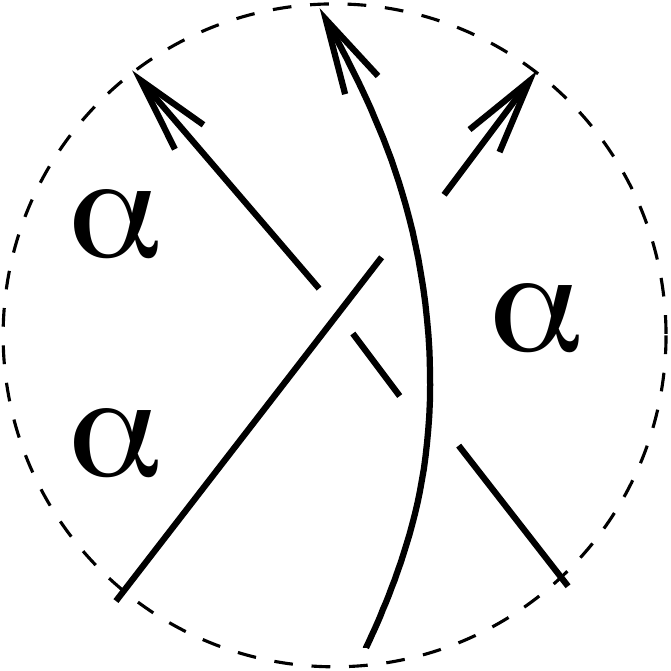}}\  \text{with canonical resolution} \ \raisebox{-13pt}{\includegraphics[height=0.45in]{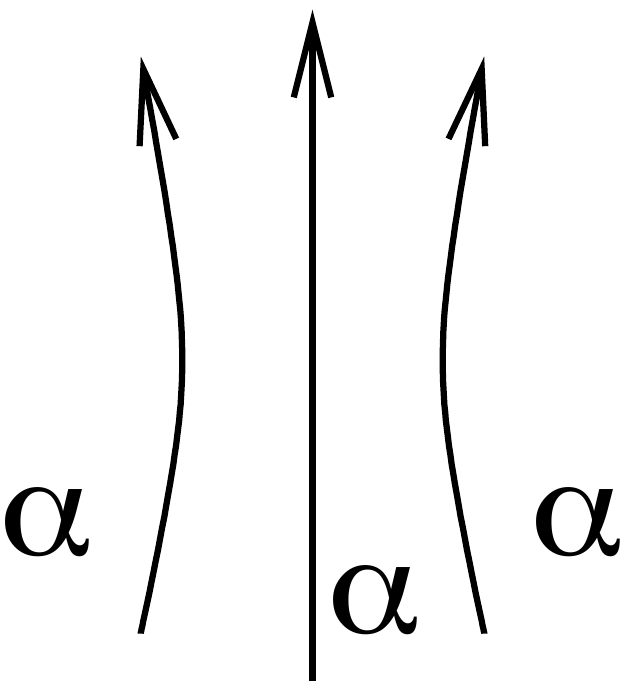}}\] 

\textit{Case 2}.  \[D_0=\raisebox{-13pt}{\includegraphics[height=0.45in]{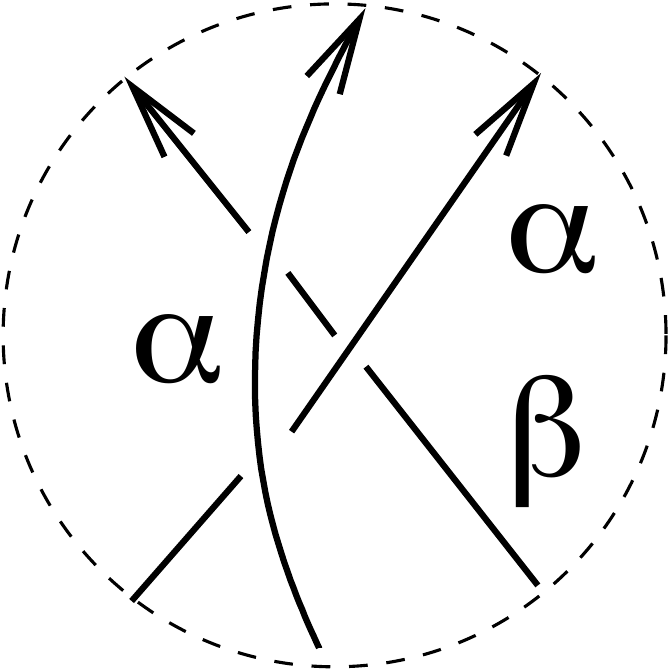}}\qquad
D_1=\raisebox{-13pt}{\includegraphics[height=0.45in]{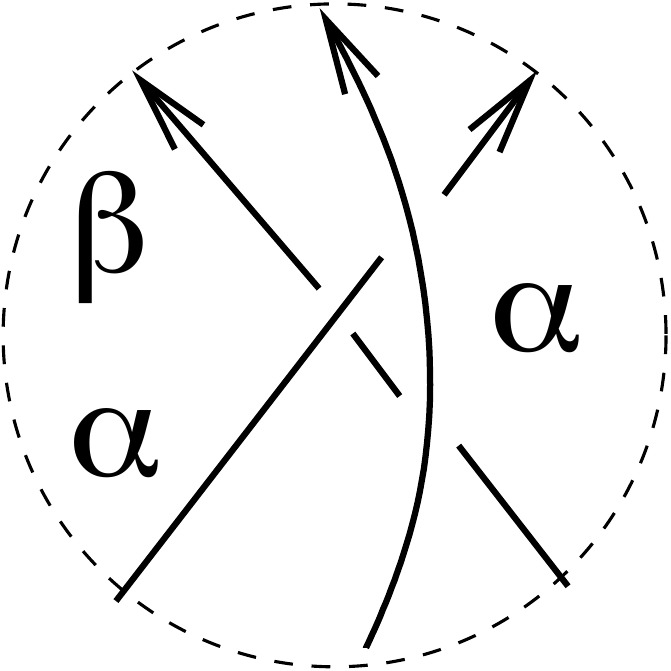}}\  \text{with canonical resolution}  \ \raisebox{-13pt}{\includegraphics[height=0.45in]{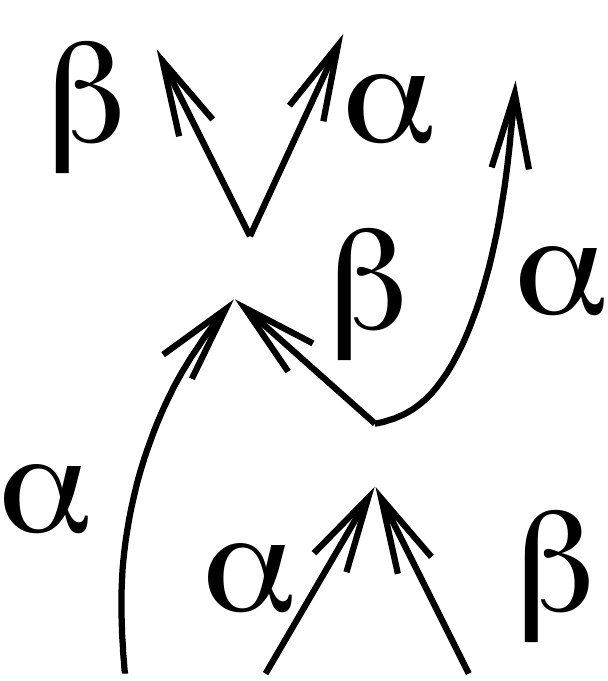}}\]

\textit{Case 3}.  \[D_0=\raisebox{-13pt}{\includegraphics[height=0.45in]{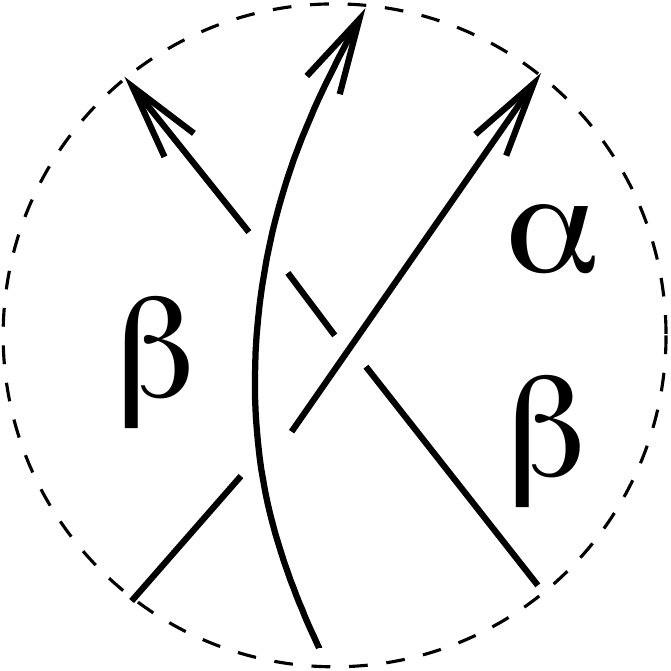}}\qquad
D_1=\raisebox{-13pt}{\includegraphics[height=0.45in]{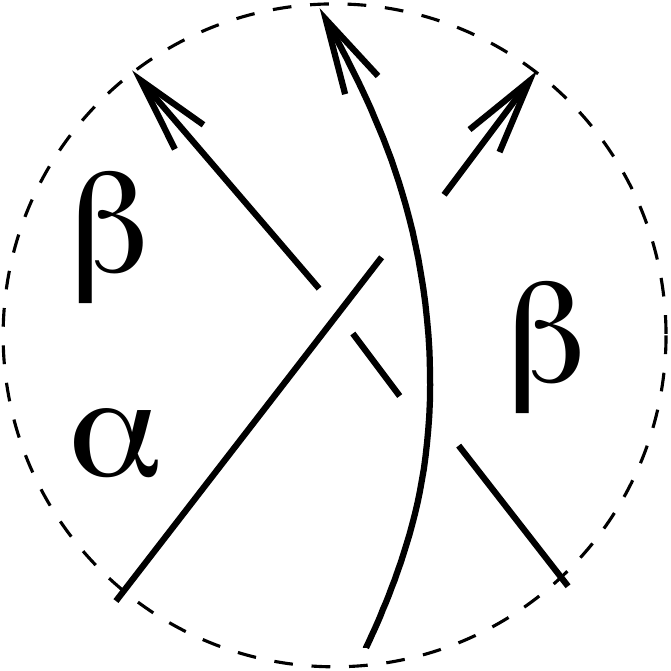}}\ \text{with canonical resolution} \ \raisebox{-13pt}{\includegraphics[height=0.45in]{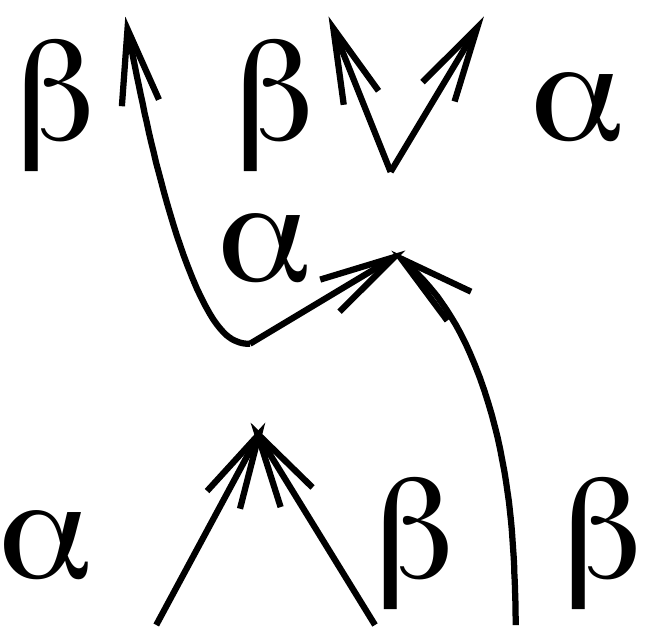}}\]

In these three cases, the canonical resolutions of $D_0$ and $D_1$ are the same and they have the same degree shift in the corresponding formal complexes $[D_0]$ and $[D_1].$ Therefore the map between the canonical resolutions of the two diagrams, constructed in the proof of the invariance theorem, is a degree zero automorphism of that (same) resolution, and thus it is $i^k \id,$ for some $k \in \{ 0, 1, 2, 3 \}.$ Note that we used here that the corresponding foam (as a map) cannot contain dots if it has no closed components, because a dot increases the degree of a map, and if it does contain closed components then those are replaced by a numerical value using the local relations $\ell.$  

\textit{Case 4}.  $D_0=\raisebox{-13pt}{\includegraphics[height=0.45in]{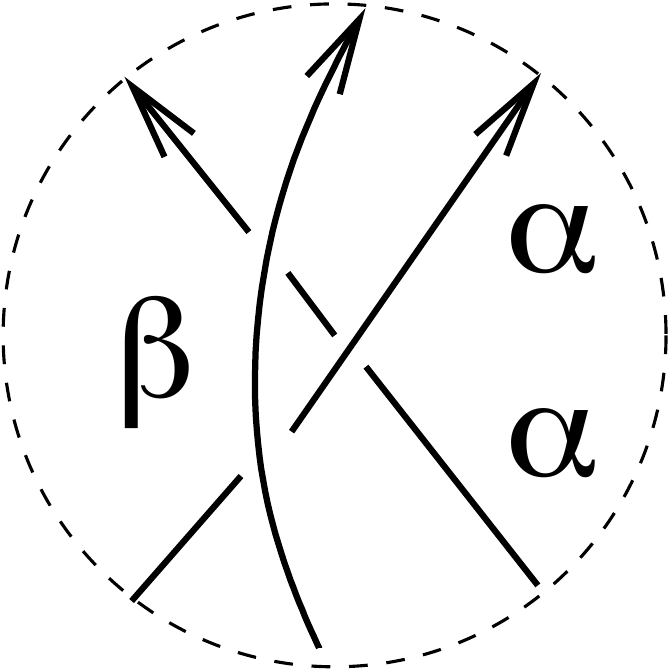}}\qquad
D_1=\raisebox{-13pt}{\includegraphics[height=0.45in]{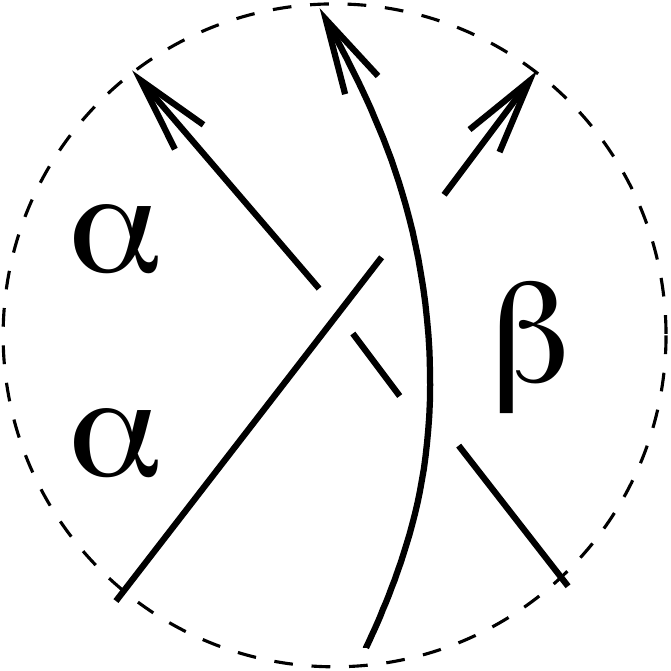}}$ whose corresponding canonical resolutions are $\raisebox{-13pt}{\includegraphics[height=0.45in]{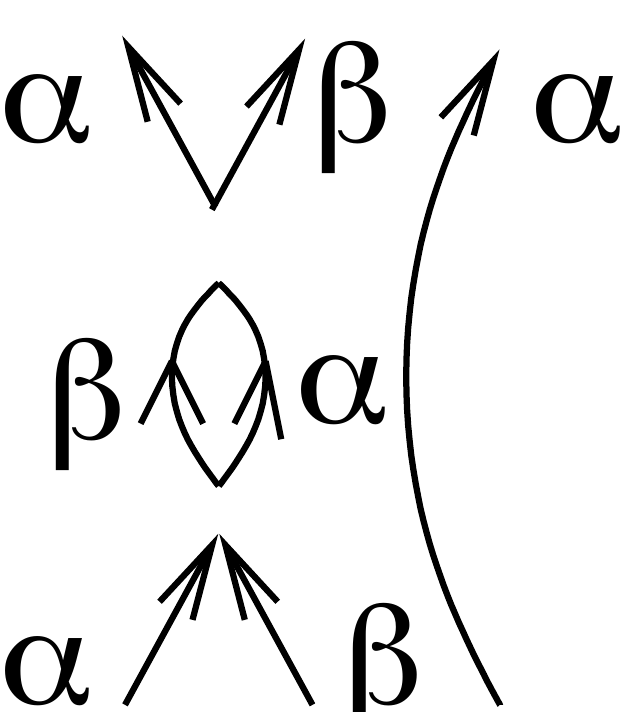}}$ and $\raisebox{-13pt}{\includegraphics[height=0.45in]{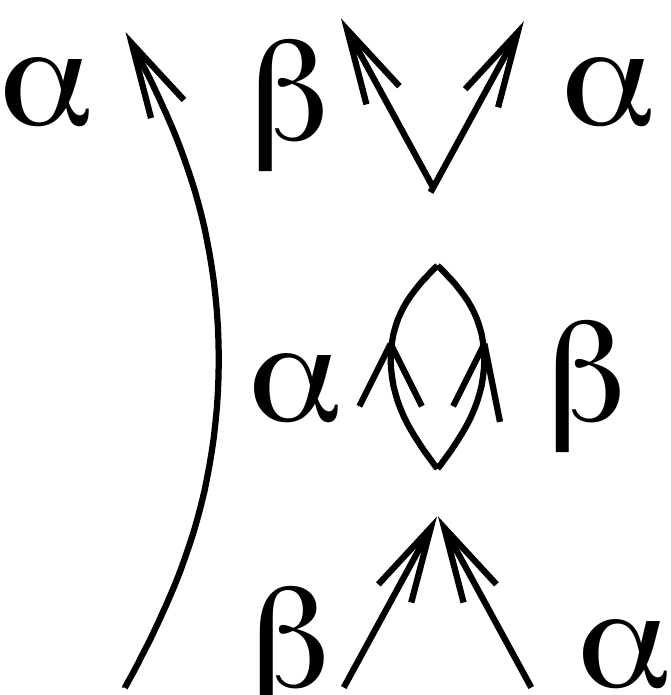}}.$ From the proof of the invariance theorem we know that the map between these resolutions is $i^k \left (\  \raisebox{-13pt}{\includegraphics[height=0.45in]{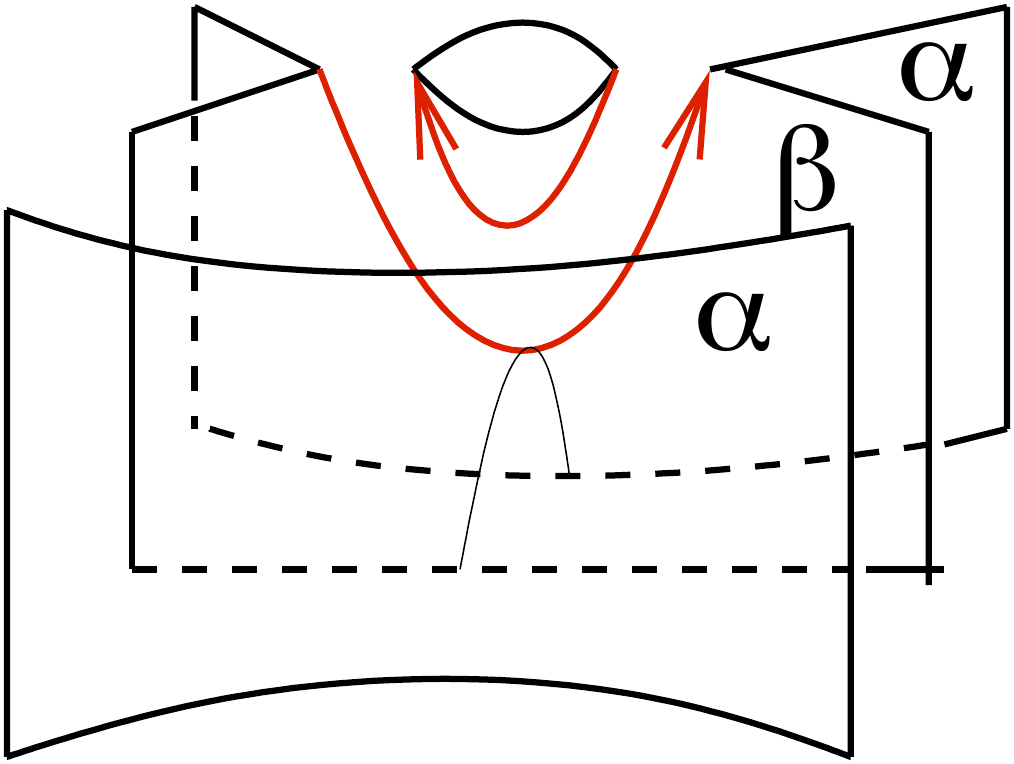}} \circ \raisebox{-13pt}{\includegraphics[height=0.45in]{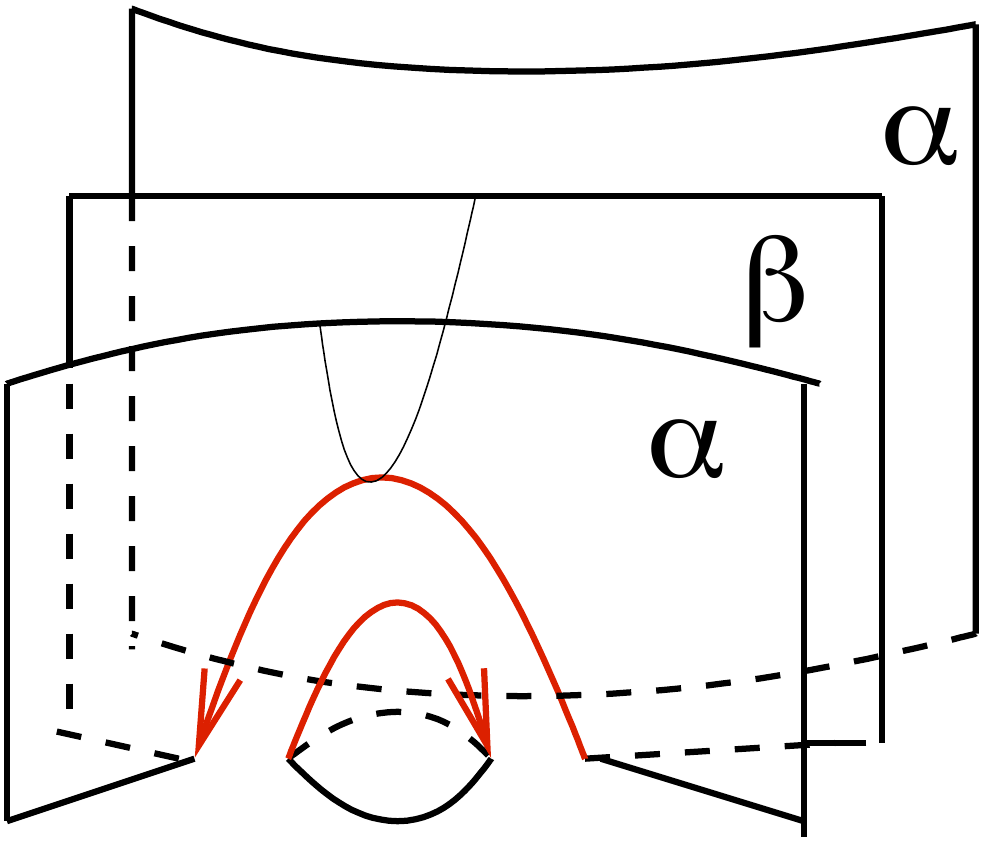}}\right ),$ for some $k \in \{0, 1, 2, 3 \}.$ 
Hence, in each of the four cases above, the induced map $\calL_C$ sends the canonical generator $h_{\phi_0}$ to a nonzero multiple of the compatible canonical generator $h_{\phi_1}.$ The other oriented versions for this move follow similarly. \end{proof}

%%%%%%%%%%%%%%%%%%%%%%%%%%%%%%%%%%%%%%%%%%%%%%%%%%%
%MORSE MOVES

\begin{proposition}\label{prop: Morse moves}
Let $C$ be an elementary link cobordism from $D$ to $D',$ and let $\calL_C \co \H(D)  \longrightarrow \H(D')$ be the homomorphism induced by $C.$

(1) If $C$ corresponds to a circle creation, then any canonical state $\phi$ of $D$ induces two canonical states $\phi'_1$ and $\phi'_2$ on $D'$ which agree with $\phi$ on all components of $D$ other than the new one, namely the disk bounded by the new circle, and $\calL_C(h_{\phi}) = \gamma_{\phi}(h_{\phi'_1} + h_{\phi'_2}), $ for some $\gamma_{\phi} \in \bbC^*.$  

(2)  If $C$ corresponds to a circle annihilation, then any canonical state $\phi$ of $D$ induces a canonical state $\phi'$ on $D'$ that agrees to $\phi$ on all components of $D,$ and $\calL_C(h_{\phi}) = \lambda_{\phi} h_{\phi'},$ for some $\lambda_{\phi} \in \bbC ^*.$

(3) If $C$ corresponds to a saddle move, then there are two situations to consider.
\begin{itemize}
\item If the values of a canonical state $\phi$ of $D$ on the two strands involved in the move are different, then $\calL_C(h_{\phi}) = 0.$
\item If the values of a canonical state $\phi$ of $D$ on the two strands involved in the move are equal, then $\phi$ induces a canonical state $\phi'$ of $D'$ which agrees with $\phi$ on the components which are unchanged and takes the common value on the changed component (or components), and $\calL_C(h_{\phi}) = \mu_{\phi} h_{\phi'},$ for some $\mu_{\phi} \in \bbC ^*.$
\end{itemize}
\end{proposition}

\begin{proof}
In the first case, $h_{\phi} \otimes z_1$ and $h_{\phi} \otimes z_2$ are the induced canonical generators corresponding to the two possible states of $D'.$ Since $\iota(1) = z_1 + z_2$ we have that $\calL_C(h_{\phi}) = h_{\phi} \otimes (z_1 + z_2)$ and the statement follows. In the second case the statement is implied by  $\epsilon (z_1) = \frac{1}{\beta - \alpha}, \, \epsilon (z_2) = \frac{1}{\alpha - \beta},$ and for the third case the statement follows from the rules for multiplication and comultiplication maps given with respect to the basis $\{z_1, z_2\}$ of $\calA.$ 
\end{proof}

Let $C$ be a link cobordism from link $D$  to link $D'$ decomposed into a union of elementary cobordisms $C = C_1 \cup C_2 \cup \dots \cup C_l.$ Define $\overline{C}_j = C_1 \cup \dots \cup C_j,$ the cobordism from $D_0 = D$ to $D_j$, where $D_{j-1}$ and $D_j$ are the initial and terminal ends of $C_j.$ Define $\overline{\calL}_j = \calL_{\overline{C}_j } = \calL_{C_j} \circ \dots \circ \calL_{C_1}.$ A component of $D_j$ is called evolved if it is a boundary component of a component of $\overline{C}_j$ that has a boundary in $D_0.$  Let $\phi_0$ be a state of $D_0.$ A state $\phi_j$ of $D_j$ is called compatible with $\phi_0$ if it agrees with $\phi_0$ on all evolved components.  

\begin{proposition}
Let $C $ be a link cobordism as above, such that $C$ has no closed components. Then
$ \overline{\calL}_j(h_{\phi_0}) = \sum_{\phi_j} \lambda_{\phi_j}h_{\phi_j},$
where $\{\phi_j\}$ runs through all states of $D_j$ compatible with $\phi_0,$ and each scalar $\lambda_{\phi_j}$ is nonzero.  
\end{proposition}

\begin{proof}
The proposition is proved by induction on $j$ and uses the results in Propositions~\ref{prop: Reidemeister moves} and~\ref{prop: Morse moves}. The proof is similar in spirit to that of~\cite[Proposition 5.10]{W}, thus we omit the details.
\end{proof}

\begin{corollary}
If $C$ is a connected cobordisms between knots $K_0$ and $K_1,$ then $\calL_C \co \H(K_0) \to \H(K_1)$ is an isomorphism.
\end{corollary}

\textbf{Acknowledgements.} The author gratefully acknowledges NSF partial support via grant DMS - 0906401. She would also like to thank the referee for his or her helpful comments.


\begin{thebibliography}{999}
\bibitem{BN0} D. Bar-Natan, {\em On Khovanov's categorification of the Jones polynomial}, Algebr. Geom. Topol. \textbf{2} (2002), 337--370; arXiv:math.QA/0201043.
\bibitem{BN1} D. Bar-Natan, {\em Khovanov's homology for tangles and cobordisms}, 
 Geom.Topol. \textbf{9} (2005), 1443--1499; arXiv:math.GT/0410495.
  \bibitem{CC1} C. Caprau, {\em The universal $sl(2)$ cohomology via webs and foams}, Topology and Its Applications \textbf{156} (2009), 1684--1702; arXiv:math.GT/0802.2848.
\bibitem{Kh1} M. Khovanov, {\em A categorification of the Jones polynomial}, Duke Math. J. \textbf{101} (2000) no. 3, 359--426; arXiv:math.QA/9908171.
\bibitem{Kh2} M. Khovanov, {\em Link homology and Frobenius extensions}, Fund. Math. \textbf{190} (2006), 179-190; arXiv:math.QA/0411447.
\bibitem{L} E. S. Lee, {\em An endomorphism of the Khovanov invariant}, Adv. Math. \textbf{197} (2005), no. 2, 554-586; arXiv:math.GT/0210213.
\bibitem{Lo} A. Lobb, {\em A slice genus lower bound from $sl(n)$ Khovanov-Rozansky homology}, Adv. Math. \textbf{222}, Issue 4 (2009), 1220-1276; arXiv:math.GT/0702393.
\bibitem{R} J. Rasmussen, {\em Khovanov homology and the slice genus}, Invent. Math. \textbf{182}, Number 2 (2010), 419-447; arXiv:math.GT/0402131.
\bibitem{T} P. Turner, {\em Calculating Bar-Nathan's characteristic two Khovanov homology}, J. Knot Theory  Ramifications, \textbf{15} Issue 10 (2006), 1335-1356; arXiv:math.GT/0411225.
\bibitem {We} C. Weibel, {\em An introduction to homological algebra}, Cambridge Studies in Advanced Mathematics, \textbf{38}. Cambridge University Press, Cambridge, 1994.
\bibitem{W} H. Wu, {\em On the quantum filtration of the Khovanov-Rozansky cohomology}, Adv. Math. \textbf{221} (2009), no. 1, 54-139; arXiv:math.GT/0612406.

\end{thebibliography}
\end{document}